\documentclass[12pt]{amsart}

\usepackage{amssymb}

\calclayout

\usepackage{color}%,ulem}

\usepackage{subfig}

\numberwithin{equation}{section}
\theoremstyle{plain}
\newtheorem{theorem}[equation]{Theorem}

\newtheorem{lemma}[equation]{Lemma}
\newtheorem{corollary}[equation]{Corollary}

\theoremstyle{remark}
\newtheorem{remark}[equation]{Remark}

\theoremstyle{definition}
\newtheorem{definition}[equation]{Definition}
\newtheorem{defn}[equation]{Definition}

\newtheorem*{question*}{Question}

\def\vint_#1{\mathchoice%
        {\mathop{\kern 0.2em\vrule width 0.6em height 0.69678ex depth -0.58065ex
                \kern -0.8em \intop}\nolimits_{\kern -0.4em#1}}%
        {\mathop{\kern 0.1em\vrule width 0.5em height 0.69678ex depth -0.60387ex
                \kern -0.6em \intop}\nolimits_{#1}}%
        {\mathop{\kern 0.1em\vrule width 0.5em height 0.69678ex
            depth -0.60387ex
                \kern -0.6em \intop}\nolimits_{#1}}%
        {\mathop{\kern 0.1em\vrule width 0.5em height 0.69678ex depth -0.60387ex
                \kern -0.6em \intop}\nolimits_{#1}}}
\def\vintslides_#1{\mathchoice%
        {\mathop{\kern 0.1em\vrule width 0.5em height 0.697ex depth -0.581ex
                \kern -0.6em \intop}\nolimits_{\kern -0.4em#1}}%
        {\mathop{\kern 0.1em\vrule width 0.3em height 0.697ex depth -0.604ex
                \kern -0.4em \intop}\nolimits_{#1}}%
        {\mathop{\kern 0.1em\vrule width 0.3em height 0.697ex depth -0.604ex
                \kern -0.4em \intop}\nolimits_{#1}}%
        {\mathop{\kern 0.1em\vrule width 0.3em height 0.697ex depth -0.604ex
                \kern -0.4em \intop}\nolimits_{#1}}}

\usepackage{graphicx}

\usepackage{url}

\newcommand{\fint}{\avint}

\title[Isoperimetric inequality for non-self-similar Sierpi\'nski sponges]{Isoperimetric and Poincar\'e inequalities on non-self-similar Sierpi\'nski sponges:\ the borderline case}
\date{\today}

\author{Sylvester Eriksson-Bique}
\thanks{S.E.-B.~was partially supported by
NSF grant DMS\#--1704215 (U.S.) and Finnish Academy grant \#--345005.}
\address{S.E--B.: Research Unit of Mathematical Sciences,
P.O.Box 3000,
FI-90014 Oulu, Finland}
\email{sylvester.eriksson-bique@oulu.fi}

\author{Jasun Gong}
\address{J.G.: Department of Mathematics, Fordham University, 441 E.\ Fordham Road, Bronx, NY 10458-5165, United States}
\email{jgong7@fordham.edu}


\newcounter{prob}
\newcommand{\N}{\ensuremath{\mathbb{N}}}
\newcommand{\R}{\ensuremath{\mathbb{R}}}

\newcommand{\cD}{\ensuremath{\mathcal{D}}}

\newcommand{\nseq}{\ensuremath{{\bf n}}}

\newcommand{\diam}{\operatornamewithlimits{ \rm diam}}

\newcommand{\defeq}{\mathrel{\mathop:}=}

\newcommand{\mT}{\ensuremath{\mathcal{T}}}
\newcommand{\mR}{\ensuremath{\mathcal{R}}}
\newcommand{\nmR}{\ensuremath{\overline{\mathcal{R}}}}

\newcommand{\Ha}{{\mathcal H}}

\def\Xint#1{\mathchoice
{\XXint\displaystyle\textstyle{#1}}%
{\XXint\textstyle\scriptstyle{#1}}%
{\XXint\scriptstyle\scriptscriptstyle{#1}}%
{\XXint\scriptscriptstyle\scriptscriptstyle{#1}}%
\!\int}
\def\XXint#1#2#3{{\setbox0=\hbox{$#1{#2#3}{\int}$ }
\vcenter{\hbox{$#2#3$ }}\kern-.58\wd0}}
\def\avint{\Xint-}

\newcommand{\co}{\mskip0.5mu\colon\thinspace}   % Colon for maps.

\begin{document}

\maketitle

%
%Contributions:
% - MacKay-Tyson-Wildrick in higher generality
% - Synthetic and general proof
% - New characterization of PI and exponent with a sufficiently good estimate at a single level
% - Large families of planar PI and Loewner carpets. General sufficient conditions.
% - New way of proving 1,1-PI via isoperimetry of ``substantial sets'' and estimates at given levels.
% - Sharp characterization and proof that works in all dimensions for square carpets/ cube sponges. New examples of spaces with $(1,1)-PI$
% - Heisenberg examples of subsets with empty interior. New geometries for PI-spaces.
% - Metric conditions, and notions of ``Metric PI-sponges''
% - Tools of independent interest, such as iteration, isoperimetry, building curves via iteration, Poincar\'e doesn't see small sets.

\begin{abstract}
%\dotfill
%
%\noindent
In this paper we construct a large family of examples of subsets of Euclidean space that support a $1$-Poincar\'e inequality yet have empty interior.  
These examples are formed from %an iterative removal process involving domains.
{%\color{blue}
%The construction uses 
an iterative process that involves removing well-behaved domains, or more precisely, domains whose complements are uniform 
in the sense of Martio and Sarvas.
}

% for a large class of (generalized) non-self-similar Sierpi\'n\-ski sponges.}
%for higher-dimensional analogues of carpets, called . 
%
%These examples subsume the previous results of Mackay, Tyson, and Wildrick and extend them to many more general shapes as well as higher dimensions.  
%Further, we introduce a scheme based on iterative removal processes involving domains whose complements are uniform.
%({\tt MOVE EARLIER? OR REMOVE})

While existing arguments rely on explicit constructions of Semmes families of curves, we 
include a new way of obtaining Poincar\'e %an isoperimetric 
inequalities %on a space 
through the use of relative isoperimetric inequalities, after Korte and Lahti. To do so, we further introduce 
%a new tool 
the notion of 
of isoperimetric inequalities at given 
density
levels and a way to iterate such inequalities. These tools are presented and apply to general metric measure measures.  %In the general setting of metric measure spaces, we show they can be iterated to obtain a full isoperimetric inequality. 

Our examples subsume the previous results of Mackay, Tyson, and Wildrick
regarding non-self similar Sierpi\'nski carpets, and extend them to many more general shapes as well as higher dimensions.

\end{abstract}

\section{Introduction}

A $p$-Poincar\'e inequality, in the sense of \cite{heinonen1998quasiconformal}, captures the notion of possessing many 
(rectifiable) 
curves in a space that connect prescribed pairs of points -- an idea made precise in \cite{semmescurves, keith2003modulus} for example. 
A smaller exponent $p$ for a $p$-Poincar\'e inequality indicates a %larger richness 
richer supply
of curves, and our focus will be on the borderline case --- that is, the
$1$-Poincar\'e inequality. 

Alternatively, such inequalities are %connected with 
related to
how easy it is to separate the space by "small" sets -- 
i.e.\ the role of isoperimetry. 
Specifically, we consider the notion for {\em relative} isoperimetry, and how boundaries separate a set from its complement; for a precise formulation of these
notions,
see %Definition \ref{def:PI} and 
Section \ref{sec:isoperim}. Our context will mainly be Euclidean spaces in all dimensions $d \geq 2$, though many techniques 
on isoperimetry 
are completely general
and apply to the metric space setting. 

In passing from a given space to a subset of that space, the number of curves decreases and it becomes easier to separate 
the subset (as a space, in its own right). Thus, a subset $A\subset \R^d$ often will not support a Poincar\'e inequality. If, however, 
%the subset is 
%{\color{red}large}
%\marginpar{\tt\small rather, \\ dense?}
%enough, or alternatively, if 
one 
removes a collection of
sufficiently ``sparse'' obstacles from the underlying space, then intuitively the Poincar\'e inequality could be preserved for the subset. %in the process. 
%Yet, understanding the quantification of ``how much'' can be removed is a subtle issue.
{%\color{blue}
It is a subtle issue, however, of ``how sparse'' these obstacles can be.}
Our main result, Theorem \ref{thm:mainthm}, gives a general sufficient condition for a $1$-Poincar\'e inequality to hold for subsets of $\R^d$ arising from such a removal process.
  
In the case of $\R^2$, 
it was shown in \cite[
Theorems 1.5--1.6]{mackaytysonwildrick} that a certain family of positive (Lebesgue) measured subsets %of $\R^2$ 
satisfy $p$-Poincar\'e inequalities, despite having empty interior. Their results are remarkable,
in that they give sharp characterizations of
%since sharp characterizations are given in terms of
the range of exponents $p$ for 
which
the $p$-Poincar\'e inequality
holds. Our work here employs substantially different techniques, and yields a more general class of examples for the exponent $p=1$. See Appendix \ref{app:explicit} for a more detailed discussion.

An earlier work by the authors \cite{bigpaper} studied 
the case of exponents $p > 1$ separately and uses completely different techniques, both from here and from \cite{mackaytysonwildrick}.
%the extension of \cite{mackaytysonwildrick} for exponents $p>1$. This range of exponents also used completely different techniques. However, similar to there, 
In a similar theme, however,
our results here hinge on new sufficient conditions for a $1$-Poincar\'e inequality. Together, they form a complete generalization of the main result of \cite{mackaytysonwildrick}.  
%In both cases, one finds new techniques and characterizations of Poincar\'e inequalities that are of independent interest.%
%)

%In this paper, we make two points regarding the preservation of Poincar\'e inequalities
%{\color{blue}
%from Euclidean spaces to their subsets.%
%}
%First, we show that the results of \cite{mackaytysonwildrick} are not isolated, but part of a more general theme. Indeed, for all $d \geq 2$ we show that a large, flexible class of subsets of $\R^d$, % with their restricted metric and measure, and even 
%despite having empty interior, support a $1$-Poincar\'e inequality with respect to the restricted Euclidean metric and the Lebesgue measure. 

In our main theorem, we make a step towards understanding 
%exactly 
which removal processes are permitted, when forming these subsets.
%To get there requires a new perspective to isoperimetry%, where one requires 
To that end, we adopt a new perspective on isoperimetric inequalities. Instead of testing isoperimetry for all sets, we first require
the inequality to hold only for sufficiently ``dense'' sets, as measured by a density parameter $\tau$. The new notion of a $(\tau,C)$-isoperimetric inequality allows 
the flexibility of
proving ``sufficiently good'' estimates at every scale, which when iterated,
%produce the desired estimate. 
{%\color{blue}
leads to isoperimetry for all sets and at all densities.}
Indeed, this added flexibility allows us to consider each scale independent of others, and is the crucial tool in 
%proving our main result.
our proof.

The $(\tau,C)$-isoperimetric inequality can be further thought of as a scale-invariant weak estimate that improves itself. This idea of self-improvement is frequent in harmonic analysis and geometric analysis and has appeared, for example, in the following classical contexts. 
The results often involve mild topological assumptions and to obtain them, one iterates the relevant estimate in an appropriate case-dependent way.  %Consider, for explicitness the following examples.
\begin{enumerate}
\item Muckenhoupt $(\epsilon,\delta)$-conditions improve to $A_p$-type conditions \cite[Proposition V.4]{steinharmonic}. 
\item  Weak quasisymmetries are quasisymmetries \cite[Lemma 6.5]{vaisalaquasi}.
\item A ``balled'' Loewner condition improves to a Loewner condition \cite[Proposition 3.1]{bonkkleiner}.
\item  The Loewner condition improves to a more quantitative estimate \cite[Theorem 3.6]{heinonen1998quasiconformal}. 
\item ``Weak''-type Poincar\'e conditions at a given level, improve to true $p$-Poincar\'e inequalities for some $p>1$ \cite[Theorems 1.2 and 1.8]{sylvester:poincare}. See also \cite[Theorem 2.19]{bigpaper} for a more quantitative version.
\end{enumerate}

\subsection{Subsets arising from %iteratively 
removing obstacles}

The sets $S$ we consider arise by removing "obstacles'' $R$ from a set $\Omega$, and thus are of the form 
$$
S=\Omega \setminus \bigcup_{k \in \N} \bigcup_{R \in \mathcal{R}_{\nseq,k}} R,
$$
where $\Omega$ is a so-called ``uniform'' domain in $\R^d$ and each $\mathcal{R}_{\nseq,k}$ is a collection of ``co-uniform'' domains $R$; that is, 
each $\R^d \setminus R$ is uniform and $\partial R$ is connected.

For the %definition
precise notion
of a uniform domain, see Definition \ref{def:uniform}; for the moment, however, we note that these include convex sets with bounded eccentricity, or regions without cusps. In particular, planar 
domains whose boundaries
%regions which are given by 
are 
quasicircles are also uniform, see \cite[Remark 4.16]{bigpaper} for the definitions of a quasicircle, some references on such examples. (As a technical point, in this paper we allow uniform domains to be closed sets.)

%In \cite{bigpaIn particular, planar regions which are given by quasicircles are also such regions.per} we introduced the 
The notion of a uniformly %$\nseq$-
sparse 
collections of co-uniform domains 
%(in a general metric space) 
was introduced in \cite[Definition 4.21]{bigpaper}
%This setting 
and
forms the starting point for our analysis.
%Note that versions of the conditions below are also related to uniformization of metric carpets, see e.g. \cite{bonkuniform}.
Below, for sets $K$ and $K'$ we denote their ``distance'' by
$
d(K,K') \defeq
\inf\{ d(x,x') : x \in K, x' \in K' \}
$.
 
\begin{defn} \label{def:sparsecoll} Let $\Omega$ be a non-empty  %closed, bounded 
compact
subset %in $X$.
of $\R^d$.
Let $\nseq = \{n_k\}_{k=1}^\infty$ be a sequence of positive integers, and consider scales
$s_0=\diam(\Omega)$ and
$$
s_k = \frac{1}{n_k}s_{k-1}
$$
for $k \in \N$. 
A sequence of collections of domains $\{\mR_{\nseq,k}\}_{k=1}^\infty$ in $\Omega \subset \R^d$ forms a {\sc uniformly ($\nseq$-)sparse collection of co-uniform domains  in $\Omega$ %with small projections
} if there are constants $\delta \in (0,1)$ and $L, A > 0$ so that
for each $R \in \mR_{\nseq,k}$:\
\begin{enumerate}
 \item $R \subset \Omega$;
 \item $R$ is $A$-co-uniform and $\Omega$ is $A$-uniform;
 \item $\diam(R) \leq Ls_k $;%.
 \item $
 {%\color{red}
 d(R,\Omega^c)
 }
 \geq \delta s_{k-1}$;  
 \item for each $R' \in \mR_{\nseq,k'}$ with $k \geq k'$, then
$d(R,R') \geq \delta s_{k-1}$;
\end{enumerate}
Moreover, %the collection 
call
$\{\mR_{\nseq,k}\}$ %is called 
{\sc dense} in $\Omega$ whenever $\bigcup_{k\in\N} \bigcup_{R \in \mR_{\nseq,k}} R$ is a dense 
subset of %in 
$\Omega$.
%Lastly, put
%$$
%S_\nseq \defeq \Omega \setminus \bigcup_k \bigcup_{R \in \mathcal{R}_{\nseq,k}} R.
%$$
\end{defn}

We note that versions of these conditions also appear in the context of uniformization of metric carpets, see e.g. \cite{bonkuniform}. 

Uniform sparseness by itself does not ensure a Poincar\'e inequality, and one needs to impose a condition on $\nseq$, the %positions/
sequence of
ratios between
scales. 
In \cite{bigpaper}, it suffices to assume that $\nseq\in \ell^d$ to obtain a $p$-Poincar\'e inequality for all $p>1$. 

In the $p=1$ case, however, we will 
%need to 
also consider the projections 
of such collections
onto subspaces. 
Below, let $\pi_1, \dots, \pi_d$ be any collection of linearly independent projections of $\R^d$ onto subspaces of codimension one,
that is, the collection of the normal vectors of hyperplanes $\pi_i(\R^d)$ form a linearly independent set in $\R^d$. 
Up to a coordinate change, we will often assume that the $\pi_i$ are coordinate projections.

\begin{defn} \label{def:smallproj}
A uniformly sparse collection of co-uniform domains $\{\mR_{\nseq,k}\}_{k=1}^\infty$ is said to {\sc have small projections} if, with the same constant $L > 0$ as before, 
\begin{itemize}
\item[(6)] %{\color{red}if \color{red}$k \geq 1$} and 
For $k \in \N$,
if $r \geq  s_{k-1}$ then for each $x \in \Omega$ and $i=1, \dots, d$ it holds that
$$
\Ha_{d-1} \Big(\pi_i\bigg( B(x,r) \cap \bigcup_{R \in \mathcal{R}_{\nseq,k}} R\bigg) \Big) \leq \frac{L r^{d-1}}{n_k^{d-1}}.
$$
%and in particular, that
%or alternatively, that
%$$
%\limsup_{k\to\infty}
%\Ha_{d-1} \left(\pi_i\left( B(x,r) \cap \bigcup_{R \in \mathcal{R}_{\nseq,k}} R\right) \right) \leq \frac{L r^{d-1}}{n_k^{d-1}}.
%$$
\end{itemize}
\end{defn}

%{\color{red}
%TODO:\ 
%Check if limsup condition is enough for the proof; if so, then omit earlier condition.
%}

%Using these we specifically prove the following theorem.
With this notion, we now formulate our main result.
{%\color{blue}
%We note that the usual notion of a uniform domain $\Omega$ requires that $\Omega$ be open; however, the same condition is also well-defined for closed sets as well. 

\begin{theorem}\label{thm:mainthm} 
Fix constants $L,A\geq 1, \delta >0$ as in Definition \ref{def:sparsecoll}. Suppose $\Omega\subset \R^d$ is a %closed bounded 
compact
uniform domain and that 
$
\{\mR_{\nseq,k}\}_{k=1}^\infty
$ 
is %a collection of open sets in $\Omega$ which forms 
a uniformly $\nseq$-sparse collection of co-uniform domains in $\Omega$ that has
small projections. The set
$$
S_\nseq \defeq \Omega \setminus \bigcup_{k\in\N} \bigcup_{R \in \mathcal{R}_{\nseq,k}} R.
$$
satisfies a $1$-Poincar\'e inequality (with respect to the restricted measure and metric) if
\begin{align} \label{eq:summability}
\sum_{k=1}^\infty \frac{1}{n_k^{d-1}} < \infty.
\end{align}
\end{theorem}
{\color{red}
%Comment (S): This paragraph seems a bit technical and hard to read. Also it leaves many questions unclear. What should we do about it?
%
%Comment (J):\ What if we just remove the {\em italicised} sentence below, entirely?  The point about whether the measure is initially Lebesgue or not is subtle and not immediately relevant to our discussion.
}

\begin{remark}[Previous results]
%\marginpar{\tt\small switched order of the remarks}
As a special case
of Theorem \ref{thm:mainthm}, 
we obtain immediately a new proof %of the result of \cite{mackaytysonwildrick} 
that certain ``non-self similar Sierpi\'nski carpets'' satisfy a $1$-Poincar\'e inequality, 
as first shown in \cite{mackaytysonwildrick}.
%We moreover extend their result to higher dimensions.
{%\color{blue} 
In particular, the removed obstacles $R$ there are coordinate squares and the uniformly sparse collections of squares have small coordinate projections.%
}
(For this specific construction and other similar ones, see Appendix \ref{app:explicit}.)

In contrast to \cite{mackaytysonwildrick}, whose results apply %only to $d=2$ and 
only to %the specific rectangular 
a construction 
{%\color{blue}
involving $d$-dimensional cubes
with $d=2$,%
}
our result applies immediately in all dimensions simultaneously and allows for many variations. % of the construction. 
For instance, %example, 
one could imagine %constructing 
sets 
$S_\nseq$
where the sets $R,\Omega$ in Definition \ref{def:sparsecoll} are all circles, or as in Figure \ref{fig:sierpinskitriangle}, where they are all triangles. 
{%\color{blue}
One could even go so far as to choose randomly-sided polyhedra with the number of sides, though uniformly bounded, varying with each scale $s_k$!
}
This way, %the result and proof we provide 
our result and proof
give a flexible way to approach such results without %limiting the geometry too much. 
overly restricting the geometry.

{\color{red}\tt\small
%TODO:\ Does a blurb about Heisenberg fit here, to fulfill the discussion of 
%\\
%generality and our results?
}

We refer to \cite{bigpaper} for a more expansive discussion on the relevance of these results.%, and will next present an outline of the proof. 
\end{remark}

\begin{remark}[Necessity of conditions]
%To put this result in context, n
Note that the conditions %we give 
given in Definitions \ref{def:sparsecoll} and \ref{def:smallproj}
are %not very far from 
close to 
necessary when $d=2$. 
Indeed, in the planar case, versions of conditions (1)-(5) %and some weaker variants of (4) and (5) %from \ref{def:sparsecoll} 
are necessary (see \cite[Theorem 4.40]{bigpaper}) while without $(6)$ and the summability condition \eqref{eq:summability} %in the statement, 
one may construct counterexamples \cite[Proposition 4.1]{mackaytysonwildrick}. 

In the case of specific constructions, such as \cite{mackaytysonwildrick} and the one in Figure \ref{fig:sierpinskitriangle}, condition (6) is sharp. 
In cases lacking sufficient symmetry, however, there is a subtlety regarding
%However, in less symmetric cases, there seems to be a flexibility, which is difficult to quantify, hinging subtly on 
the precise placement of obstacles.  In a similar spirit as say, the quasiconformal Jacobian problem, it
appears very difficult %, even impossible, 
to formulate a completely sharp result; see e.g. \cite{semmesbilip} for further discussion on similar characterization problems.

In any event, some minimal assumption, say the
%To elaborate some, 
the weaker condition 
that
$\nseq\in \ell^d$, 
is needed for the set to have positive Lebesgue measure 
and guarantee the validity of some Poincar\'{e} inequality. 
%{\color{red}
%\marginpar{\tt\small can we omit this, but still say unavoidable?}
%In using the restricted measure, this is unavoidable,
%although even altering the measure would require this. This follows from the seminal result of De Philippis, Marchese, and Rindler, see \cite{philiprindler} or \cite[Theorem 1.5]{bigpaper} for a precise statement.%
%}
%
% {\em
% We remark here, that this conclusion is  stronger in the sense that $S_\nseq$ must have positive Lebesgue measure even if it were to satisfy a Poincar\'e inequality %and doubling for (potentially) some other 
% for some (potentially) different
% doubling
% measure. %than the restricted measure.
%}
%
%{\tt\color{red} TODO:\ add paragraph about (6) here.}
\end{remark}

In order to obtain Theorem \ref{thm:11PI}, we need a flexible way to prove such inequalities, and a condition which 
%views the space scale-wise.
handles such inequalities at fixed scales.
This involves a new notion of isoperimetry which applies to all metric measure spaces, not just Euclidean ones. Indeed, the only place where %we need 
the Euclidean structure 
is used
is in using projections and a projected isoperimetric inequality, see Lemma \ref{lem:projection}. It is conceivable that analogous structures exist in other settings. For example, the Heisenberg group has a natural collection of vertical projections \cite[Definition 2.2]{orponenetal}, and one may push our main result to such settings. We leave this for future exploration.

\begin{figure}%
    \centering
    \subfloat[Step one: Removal of the central white triangle (relative to height) with $n_1=5$.]{{\includegraphics[width=7cm]{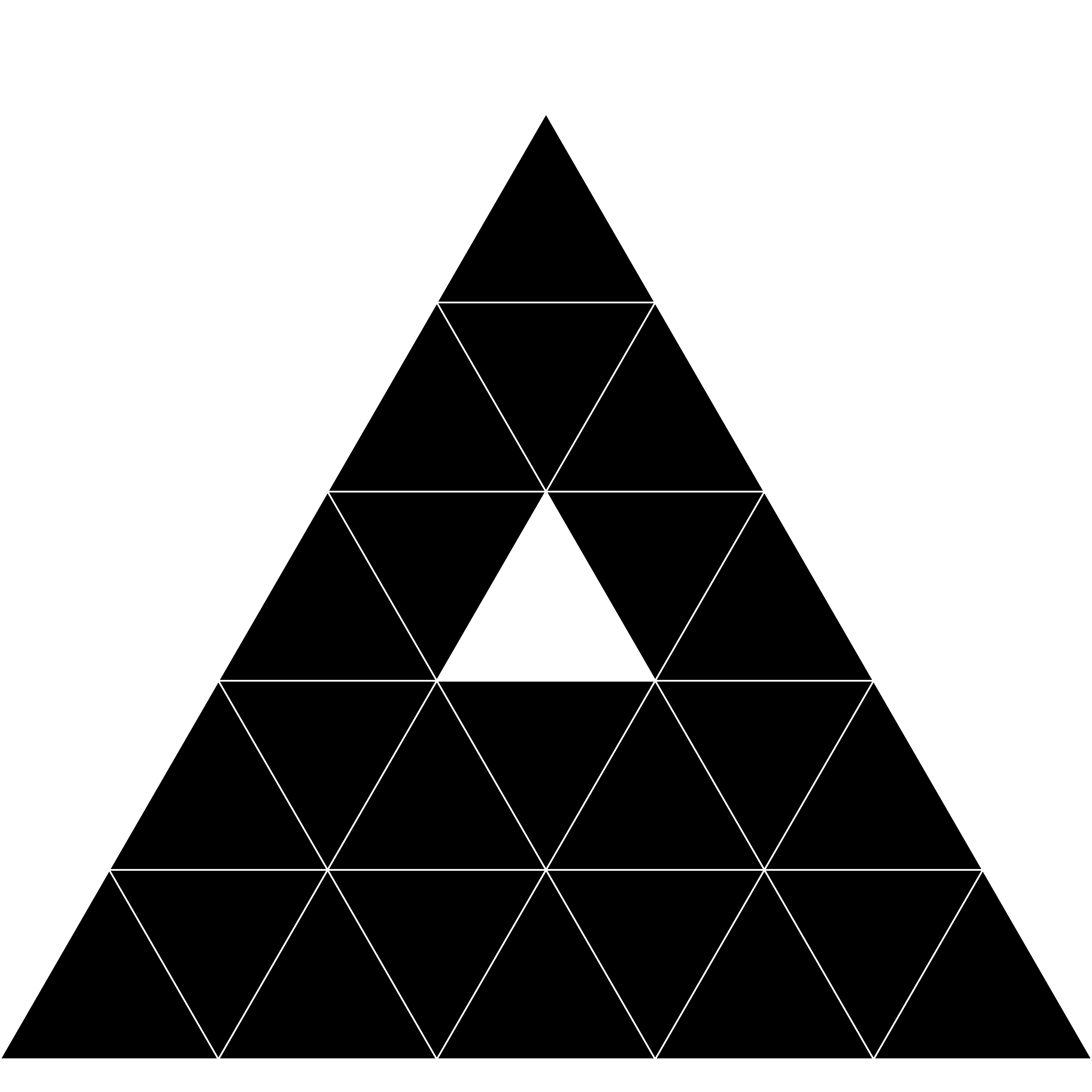} }}%
    \qquad
    \subfloat[Three steps: Domain is in black, and removed triangles in white. We used $n_1=5,n_2=7,n_3=11.$]{{\includegraphics[width=7cm]{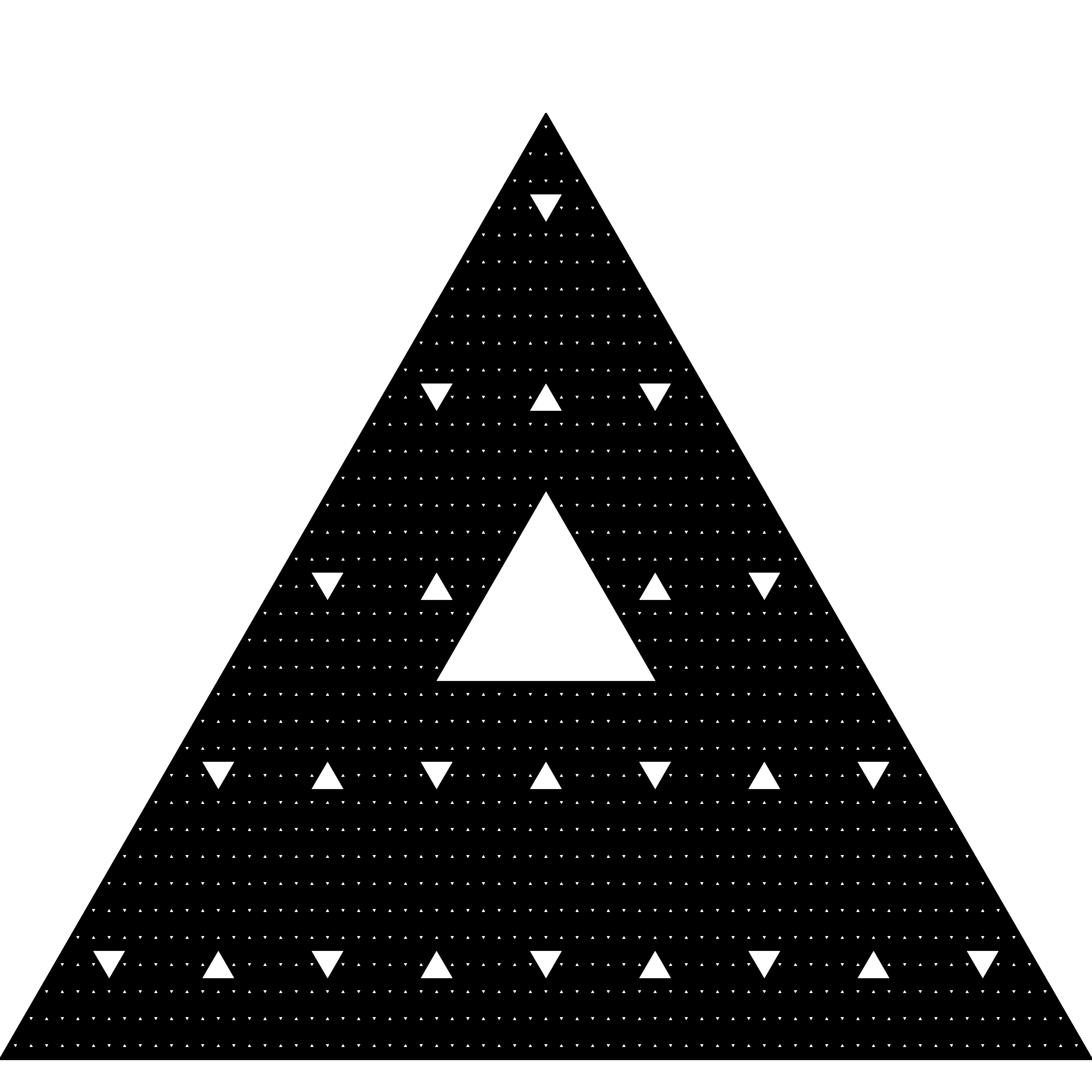} }}%
    \caption{Construction of a non-self-similar triangular version of a Sierpi\'nski carpet.}%
    \label{fig:sierpinskitriangle}%
\end{figure}

\subsection{Iterating Isoperimetry}

The case of $p=1$ is a borderline case for the Poincar\'e inequality and an inherently geometric one.
Consider, for example, the well-known correspondence between
%compare with 
the Sobolev embedding theorem and %its correspondence with 
the isoperimetric inequality. 

Related to this, we will employ the characterization of Lahti and Korte \cite{lahtikorteisoperim} which asserts that the validity of a %$1$-
$1$-Poincar\'e inequality is equivalent to a so-called ``relative isoperimetric inequality'', see Theorem \ref{thm:11PI}. %We establish Theorem \ref{thm:mainthm} by proving isoperimetry and then referring to this fact.  
Since,
this fact and 
many of the following results hold 
true in general metric measure spaces,
the forthcoming discussion will also be formulated in the context of metric measure spaces.

The isoperimetric inequality is easier to establish for ``larger'' sets in the sense of density, as defined below.

\begin{definition}
Let $(X,\mu)$ be a measure space and let $E,F$ be measurable subsets of $X$ with $\mu(F) > 0$. 
We define
%we refer to 
the quantity
%\marginpar{\tt\small\color{red} give a name to this?}
$$
\Theta_\mu(E,F) \defeq 
%\frac{\min\{\mu(E\cap B(x,r)), \mu(B(x,r) \setminus E) \}}{\mu(B(x,r))}.
\frac{\min\{\mu(E \cap F), \mu(F \setminus E) \}}{\mu(F)}.
$$
%whenever $F \subset X$ satisfies $\mu(F) > 0$.
\end{definition}

This quantity is symmetric for $E$ and its complement $E^c$. We are interested in the sets $E$ with $\Theta_\mu(E,F) \geq \tau$ for a given $\tau>0$ -- i.e. the ones that are 
neither empty nor full, quantitatively.
%roughly half-full and half-empty. 
 In our proof of Theorem \ref{thm:mainthm}, for sets with larger 
{%\color{red}
relative density $\Theta_\mu(E,F)$
}
 we can throw away and control a junk-set coming from conditions (5) and (6) in Definitions \ref{def:sparsecoll} and \ref{def:smallproj}.  Interestingly enough, to prove isoperimetry it suffices to consider such sets. More precisely, we prove the following fact, which applies to all metric measure spaces and may be of independent interest.
 Here, 
as in Definition \ref{defn:relisoconsts},
a $(\tau,C)$-isoperimetric inequality is one that only holds for sets $E \subset B=B(x,r)$ with $\Theta_\mu(E,B) \geq \tau$, %and 
where $C$ is a %constant giving the quality of such an inequality.  
multiplicative constant.
%; see Definition \ref{defn:relisoconsts} for details.

\begin{theorem}\label{thm:largeenough} Let $(X,d,\mu)$ be $D$-doubling with $D \geq 2$, and let 
$
\tau \in (0,\frac{1}{D^3}]
$.
If $X$ satisfies a $(\tau, C)$-isoperimetric inequality with inflation factor $\Lambda$, then $X$ satisfies the relative isoperimetric inequality with constants
$C_S = C_S(D,\Lambda)=D^{7+\log_2(\Lambda)}C$
and
$\Lambda_S=\Lambda_S(\Lambda)= 2\Lambda$. 
\end{theorem}

The proof of this theorem involves ``iterating'' an estimate at appropriate scales. This method of relative isoperimetry via iteration is new.  To the authors' knowledge, this is the first instance where it is used to verify a previously-conjectured Poincar\'e inequality. This method has the advantage that it allows to throw away small sets (such as those arising from condition (6) in Definition \ref{def:smallproj}). It further allows to focus on a single scale at a time.

%While these results are fully general, to obtain Theorem \ref{thm:mainthm} we still need a little Euclidean structure through an appropriate collection of projections $\pi_i$. 
%{\color{blue}
%Though Theorem \ref{thm:largeenough} applies to the general metric space setting, some properties of Euclidean space are used to obtain Theorem \ref{thm:mainthm}.  More precisely, only the product structure of $\R^d$, by way of linearly independent projections, is needed;
}
%Indeed, the only place where we need them is the projected isoperimetric inequality, 
%see Lemma  \ref{lem:projection} below.  This estimate, combined with an appropriate choice of balls and throwing away the ``shadows'' of the obstacles using Condition (6) in Definition \ref{def:sparsecoll}, establishes the $(\tau,C)$-isoperimetric inequality and thus Theorem \ref{thm:mainthm}.

%Finally, we remark, that it is even conceivable that our techniques could apply to some  non-Euclidean  examples, such as the (first) Heisenberg group. In the Heisenberg group, one has an analogous collection of ``vertical projections'', that could replace $\pi_i$, and one may hope to obtain an analogue of Lemma \ref{lem:projection}. 
%However, we leave such a development for future work. For further discussion about the Heisenberg group and other sub-Riemannian manifolds, see \cite{montgomery}. For the vertical projections, see \cite[Definition 2.2]{orponenetal}
%could be constructed if one replaced the 
%orthogonal projections
%argument with $\pi_i$ 
%by a suitable collection of families of curves with constant density. %However, then one would need a new proof of Lemma \ref{lem:projection}. 
%We do not pursue this direction here.
%{\color{blue}

\subsection{Outline}

The remainder of the paper is organised as follows. In Section  \S\ref{sec:preliminaries} we discuss preliminaries on measure and uniformity and state crucial lemmas. Section \S\ref{sec:isoperim} is devoted to facts about isoperimetry and stating the isoperimetric inequality. In that section we also prove Theorem \ref{thm:largeenough} and give the projected isoperimetric inequality in $\R^d$ in Lemma \ref{lem:projection}.

The proof of the main result (Theorem \ref{thm:mainthm}) is then left to Section \S\ref{sec:proof}. This rests on establishing the $(\tau,C)$-isoperimetric inequality for subsets $E \subset B(x,r) \cap  S_{\nseq}$. First, we reduce the case of the sum in Equation \eqref{def:smallproj} being small. This is done by localizing the argument.  Following this, one replaces $B(x,r)$  by a better ball, which does not intersect too large obstacles. To this ball one applied first Lemma \ref{lem:projection} to obtain some Euclidean boundary. The restriction on projections, and the large enough density of $E$ and its complement, means that some of this boundary must lie in the original carpet. A precise quantitative bound yields the result.
{\color{red}
%After, this we replace the original ball $B(x,r)$ by a better, and smaller, ball $B(y,s)$ which does not intersect a ``large'' obstacle set $R \in \mathcal{R}_{\nseq,k}$ for any $k$. This is done by using Corollary \ref{cor:smallerscale} and Lemma \ref{lem:containedballs}. Once such a good ball is obtained, Lemma \ref{lem:projection} gives some boundary in $\R^d$ with a large projection. This projection can not be entirely ``shadowed'' by obstacles $R \in \mathcal{R}_{\nseq,k}$ by the choice of the good ball and (6) in Definition \ref{def:smallproj}. The portion not shaded must then come from some boundary of $E$ in $S_\nseq$, and when quantified, gives a lower bound for the size of such a boundary.
}

In Appendix \ref{app:explicit} we give the explicit example of a non-self similar Sierpi\'nski sponge similar to \cite{mackaytysonwildrick}, and show how the higher dimensional generalization of their result follows from Theorem \ref{thm:mainthm}.

\section{Preliminaries}\label{sec:preliminaries}

\noindent \textbf{Notational convention:} There are many constants to keep track of in our proof. In doing so, we shall use the notation $C=C(A,B,\dots)$ to indicate when a constant $C$ in a statement depends on other constants $A,B$ in the same statement. 

\vspace{.1cm}

\subsection{Measure theoretic preliminaries} 
%\section{Intermediate results}\label{sec:intermediate}
%\section{Notation and prior results}
%
%\subsection{Notation and Basic Notions}

For the interest of generality, many of the statements in this preliminary section will be formulated for general metric measure spaces, while our main result (Theorem \ref{thm:mainthm}) is formulated only for $X=S_\nseq \subset \R^d$, where we employ the restricted measure $\mu=\lambda|_{S_\nseq}$, %and restricted metric 
with $\lambda$ the usual Lebesgue measure on $\R^d$. 
%We will assume that $\mu(B(x,r))>0$ for all balls. 

To this end, open balls in a metric space are denoted $B=B(x,r)$, and their dilations by $CB=B(x,Cr)$, despite the ambiguity that balls may not be uniquely defined by their radii. Where necessary, %and especially in the proof of Theorem \ref{thm:mainthm}, 
we will include subscripts to indicate the ambient space in which the balls are located. Thus if $X \subset \R^d$, 
then with center $x \in X$ and radius $r > 0$, the ball in $\R^d$ is $B(x,r) = B_{\R^d}(x,r)$ while the ball in $X$ is $B_X(x,r) = B(x,r) \cap X$.
%$B_X(x,r)$ is the ball in $X$, whereas $B(x,r)$ would be the ball in $\R^d$. 
We also apply this subscript notation for {\em relative boundaries}, i.e.\ $\partial_XE$ refers to the boundary of $E$, when $E$ is treated as a subset of $X$.

Throughout this paper we will %assume that the support of the measure is equal
consider only measures $\mu$ whose support is all of the underlying metric space $X$, i.e.
${\rm supp}(\mu)=X$
and that 
$\mu(B(x,r)) \in (0,\infty)$
for all balls in the metric space. 

The volume of any unit ball $B(x,1)$ in $\R^d$ is $\omega_d = \frac{\pi^\frac{d}{2}}{\Gamma(\frac{d}{2}+1)}$, where $\Gamma(x)$ is the %$\Gamma$-function.
standard Gamma function.

\begin{defn} \label{def:doubling} \label{def:ahlforsreg}
A metric measure space $(X,d,\mu)$ with a Radon measure $\mu$ is said to be \emph{\sc $D$-(measure) doubling}
if for
all $r \in (0,{\rm diam}(X))$ and any $x \in X$ we have
\begin{equation}\label{eq:doubling}
0 < \mu(B(x,2r)) \leq D \mu(B(x,r))
\end{equation}
%This condition forces $D \geq 1$. Further, 
and 
$(X,d,\mu)$ is said to be {\sc Ahlfors $Q$-regular} (with constant $C_{AR}>0$) if for
all $r \in (0,{\rm diam}(X))$ and any $x \in X$ we have
\begin{equation}\label{eq:arreg}
\frac{1}{C_{AR}} r^Q \leq \mu(B(x,r)) \leq C_{AR}r^Q.
\end{equation}
\end{defn}

\noindent
It is easy to see that the doubling condition \eqref{eq:doubling} forces $D \geq 1$ and that every
Ahlfors $Q$-regular space is %$D=D(C_{AR})=2^QC_{AR}^2$-doubling.
$2^QC_{AR}^2$-doubling. 

Further, every $D$-doubling space is metrically doubling, 
in the following sense:\
a metric space $X$ is {\sc $N$-metrically doubling} if there is a constant $N \in \N$ so that for every ball $B(x,r) \subset X$, there are %$N$
centers $x_1, x_2, \dots, x_N$ 
(possibly overlapping)
so that 
$$
B(x,r) \subset \bigcup_{i=1}^N B(x_i,r/2).
$$
The metric doubling constant 
of a $D$-doubling space %depends on $D$ by 
is
$N=N(D)=D^4$. See \cite{heinonen2000} for further details about metrically and measure doubling spaces.

Occasionally, we will assume that $X$ is geodesic, that is between any pair of points $x,y \in X$ there is a rectifiable curve $\gamma: I \to X$ that connects them and with length %$d(x,y)=\len(\gamma)$. 
$d(x,y)$.
This is automatically true for our main space of interest, $X=\R^d$. For metric measure spaces in general, it ensures that
$\mu(\partial B(x,r)) = 0,$
for any ball $B(x,r) \subset X$,
in which case the map $(x,r) \to \mu(B(x,r))$ then becomes continuous.
See, for example, \cite{buckley}. %Indeed, the map $(x,r) \to \mu(B(x,r))$ then becomes continuous.

For any subset $E \subset X$ and $x \in X$, we denote the distance to the set $E$ by 
$$
d(x,E)=\inf_{e \in E} d(x,y).
$$
(For empty sets $E$, we interpret the infimum and hence the distance as infinite.)

If $E \subset X$ is a $\mu$-measurable set, then $x \in X$ is called a point of density of $E$ if 
\begin{equation}\label{eq:pointofdens}
\lim_{r \to 0} \frac{\mu(B(x,r) \cap E)}{\mu(B(x,r))} = 1.
\end{equation}
The following lemma allows for choosing scales with density within a desired range, %-- not too large or too small. 
once an upper bound for the density is met.

\begin{lemma} \label{lem:goodscale}
Let $X=(X,d,\mu)$ be a $D$-doubling metric measure space and let $E$ be a Borel subset of $X$.  If $x$ is a point of density of $E$ and if $R>0$ is such that
$$
\frac{\mu(E \cap B(x,R))}{\mu(B(x,R))} \leq b
$$
for some $b \in (0,1)$, then there exists
$r' \in (0,R)$ such that 
$$
\frac{b}{D} \leq \frac{\mu(E \cap B(x,r'))}{\mu(B(x,r'))} \leq b.
$$
\end{lemma}

\begin{remark}\label{rmk:continuity}
Above,
if $X$ is such that the map $(x,r) \to \mu(B(x,r))$ is continuous, then the upper bound is attained for some $r'$, i.e.\
$$
\frac{\mu(E \cap B(x,r'))}{\mu(B(x,r'))}=b.
$$
This occurs, for instance, when $X$ is a geodesic metric space,  or when $X$ is a subset of a geodesic metric space and equipped with the restricted measure and distance.
\end{remark}

\begin{proof}[Proof of Lemma \ref{lem:goodscale}]
Put $h(t)\defeq \frac{\mu(E \cap B(x,t))}{\mu(B(x,t))}$, %. We have 
so
$\lim_{t \to 0} h(t)=1$, %If $X$ is geodesic, then $h$ is continuous and the claim follows by the intermediate value theorem. If not, then from doubling we have that $\frac{h(t/2)}{h(t)} \leq D$ for all $t > 0$.  
and put $R_k = 2^{-k}R$. Since $x$ is a point of density of $E$, there is some $N_0 \in \N$ such that for all $k \geq N_0$ we have $h(R_{k+1})>b$. Let $K$ be the largest index such that $h(R_K)\leq b$. 

From doubling we have %that
$\frac{h(t/2)}{h(t)} \leq D$
for all $t > 0$.  
So if 
$h(R_K) \leq D^{-1}b$, then 
$$
h(R_{K+1}) = 
h\Big( \frac{R_K}{2} \Big)\leq 
D \cdot h(R_K) \leq b
$$
which is a contradiction to $K$ being the largest index with this property. %Thus we have 
The desired estimate for $r'=R_K$
follows.
\end{proof}

%We now proceed to state two useful
%begin with checking 
%properties of the Lebesgue measure $\lambda$ restricted to $S_{\nseq}$. These statements have been proven in \cite[Section 4]{bigpaper}.

\subsection{Preliminaries on Uniformity} \label{sec:uniformity}

Here, a curve is a continuous map $\gamma\colon I \to X$ and $\Omega \subset X$ will denote a closed set. 

\begin{definition}\label{def:uniform} 
Given $x,y \in \Omega$ and $A \geq 1$, we say that $\gamma\colon[0,1] \to X$ is an {\sc $A$-uniform curve} between $x$ and $y$ in $\Omega$ if $\gamma(0)=x, \gamma(1)=x$, $\diam(\gamma) \leq A d(x,y)$, and 

\begin{equation}\label{eq:uniformity}
d(\gamma(t),\Omega^c) \geq \frac{1}{A} \min\{\diam(\gamma|_{[0,t]}),\diam(\gamma|_{[t,1]})\}.
\end{equation}
for all $t \in [0,1]$.

A subset $\Omega \subset X$ is called {\sc $A$-uniform} if for every $x,y \in \Omega$ there is an $A$-uniform curve between $x$ and $y$. If $\Omega^c = \emptyset$, we apply the convention $d(x,\emptyset)=\infty$, and the condition is vacuously satisfied for $\Omega = X$.

An open subset $R$ of $X$ is called ($A-$)\emph{\sc co-uniform} if $\partial R$ is connected and $X \setminus R$ is $A-$uniform.
\end{definition}

%{\color{blue}
The notion of uniform domains $\Omega$ has been 
introduced by Martio and Sarvas, \cite{martiosarvas} where it was initially required 
that such sets $\Omega$ be open. 
We remark, that if a 
closed subset $\Omega$ is uniform then its interior ${\rm int}(\Omega)$ is uniform in 
the classical sense. One can also show that $\partial \Omega$ is a porous subset, and thus 
has zero Lebesgue measure. 
This allows us to apply many of the calculations in classical literature, see \cite[Section 4.2]{bigpaper} for a more detailed discussion.
We also refer to \cite{vaisala1988,  bjornuniform} for further, fundamental results about such domains.
Co-uniformity, which appears in \cite{bigpaper}, was introduced as a further, convenient context for multiply-connected domains.

%Sets whose complements are uniform are referred to as co-uniform. More specifically, 

Clearly every set is a uniform domain in itself, i.e.\
%We further note that 
if $\Omega=X$ 
then
the conditions 
in Definition \ref{def:uniform} are 
automatically satisfied.
%There are two lemmas we will also need.
Of the next two lemmas, the first relates uniformity of domains to the previous notion of doubling (Definition \ref{def:doubling}) and the second is a technical estimate to be used later.
See \cite[Lemma 4.24]{bigpaper} 
and \cite[Lemma A.3]{bigpaper}
for detailed arguments, %which rely on the paper 
respectively. {%\color{red}
The first of these also follows easily from 
\cite[Lemma 4.2]{bjornuniform}.}

%We briefly need the following lemma, whose proof is contained in \cite[Lemma A.3]{bigpaper}.

\begin{lemma}\label{lem:arreg} Suppose $\Omega \subset X$ is $A$-uniform and $A$ and that $X$ is {%\color{red}
$Q$-Ahlfors regular with constant $C$}, then $\Omega$ is {%\color{red}
$Q$-Ahlfors regular with constant $C'= C'(C,Q)$} %regular 
when equipped with the restricted measure.
\end{lemma}

\begin{lemma} \label{lem:basicuniflemma} Let $\Omega \subset X$ be a closed subset and let $x,y \in \Omega$.  If $\gamma \co [0,1] \to \Omega$ is an $A$-uniform curve between $x$ and $y$ in $\Omega$, then for every $t \in [0,1]$ it holds that
$$
 d(\gamma(t), \Omega^c) \geq \frac{1}{4A}\min\{d(x, \Omega^c) + \diam(\gamma|_{[0,t]}),d(y, \Omega^c) + \diam(\gamma|_{[t,1]})\}.
$$
\end{lemma}

We will also need the following 
result.  It affirms the intuitive idea that nontrivial overlaps of uniform domains are also uniform.
{\color{blue}
}%technical result from 
The proof is somewhat technical, however, and can be found in
\cite[Theorem 4.22]{bigpaper}.

\begin{theorem}\label{thm:cutout}
Fix structural constants $A_1,A_2,C,D \geq 1$.
Let $X$ be a $C$-quasiconvex, $D$-metric doubling metric space, let $\Omega$ be an $A_1$-uniform subset of $X$, and let $S$ be a bounded, %subset of $\Omega$ that is 
$A_2$-co-uniform subset of $X$.  If
$$
\overline{S} \subset {\rm int}(\Omega)
%N_{\epsilon{\rm diam}(S)}(S) ~\subset~ {\rm int}(\Omega)
$$
then $\Omega \setminus S$ is %a 
$A'$-uniform %domain 
in $X$,
with dependence $A' = A'(A_1,A_2,C,D, \frac{d(S,\partial \Omega)}{\diam(S)})$.
\end{theorem}

We will also need the following ``collared'' estimate for neighborhoods of uniform domains.
%If $E \subset X$ is a non-empty subset, then 
In general, for non-empty subsets $E$ of $X$
we define

\begin{equation}\label{eq:neighborhood}
N_r(E) = \bigcup_{e \in E} B(e,r).
\end{equation}

\begin{lemma} \label{lem:neighborhood} 
Let $A,D \geq 1$ and let $(X,d,\mu)$ be a $D$-%measure
doubling space. There are constants $C_N=C_N(D,A) \geq 1$ and $b=b(D,A) > 0$ %\in (0,\infty)$
%Suppose that $R \subset X$ is non-empty $A$-uniform and $x \in R$, then we have
so that 
for every nonempty $A$-uniform subset $U$ of $X$, 
every $x \in U$, every $r \in (0,{\rm diam}(U)]$, and every $\delta \in (0,1)$
it holds 
that
$$
\mu(B_U(x,r))
%\mu(B(x,r) \cap U
\cap N_{\delta r}(U^c)) \leq C_N \delta^b \mu(%B(x,r)\capU
B_U(x,r)
).
$$
\end{lemma}

This is part of \cite[Theorem 2.8 \& Lemma 4.2]{bjornuniform}
in which uniform domains are known to satisfy the so-called ``corkscrew'' and ``local $b$-shell'' conditions for balls. While, the proofs there are %for only open uniform sets, they apply the same 
formulated for uniform domains that are open, they apply in our setting of closed uniform domains too,
since $\mu(\partial R)=0$ as concluded above. 
The following lemma allows us to exchange balls in an appropriate way for ones which
%sets that 
are strictly contained inside %the subset 
$\Omega$.

\begin{lemma}\label{lem:containedballs} 
Let $A,D \geq 1$, let $X$ be a geodesic $D$-doubling space, and let $\Omega \subset X$ be $A$-uniform. For any $\eta \in (0,1)$ there is a $\sigma=\sigma(D,A,\eta) \in (0,1)$ so that for any choice of
$x \in \Omega$, $r>0$, and $E \subset \Omega$
with
%if $E$ satisfies
$$
\Theta_\mu(E,B(x,r)) \geq \eta,
$$
%and if $s \in (0,\sigma)$, then 
%then for each 
and any 
$s \in (0,\sigma)$
there exists $y \in B(x,4Ar)$ so that $d(y, \Omega^c) \geq \sigma  r$ and 
$$
\Theta_\mu(E,B(y,s r)) \geq \frac{1}{2D^2} 
\Theta_\mu(E,B(x,r)) \geq \frac{1}{2D^2} \eta.
$$
\end{lemma}
\begin{proof} 
%Fix $B(x,r), E \subset \Omega$ and $\eta$, the doubling constant $D$ and uniformity constant $A$.
With the same constants as above,
let $b=b(D,A)$ and $C_N=C_N(D,A) \geq 1$
be as in Lemma \ref{lem:neighborhood}.
Now put
%Choose 
$$
\sigma = %\sigma(D,A,\eta)= 2^{-4}A^{-1}\left(\eta 2^{-1}C_N^{-1}\right)^{1/b} \in (0,1)
\frac{1}{16A}
\left(\frac{\eta}{2 C_N}\right)^{1/b},
$$
so $\sigma \in (0,1)$. 
%Pick any $s \in (0,\sigma)$.
Let $s \in (0,\sigma)$ be given, and 
in what follows, 
put $\delta \defeq 8A\sigma < \frac{1}{2} (\frac{\eta}{2 C_N})^{1/b}$ and let 
\begin{align*}
F_i=
\begin{cases}
E \setminus N_{2\delta r}(\Omega^c),&
\text{if } i = 1
\\
B(x,r) \setminus (E \cup N_{2\delta r}(\Omega^c)),&
\text{if } i = 2
\end{cases}
\end{align*}
%$F=E \setminus N_{8As r}(\Omega^c)$ or $F=B(x,r) \setminus (E \cup N_{8As r}(\Omega^c))$ in the following calculations. 
{%\color{blue}
It follows from our hypothesis that
%%we have $\Theta_\mu(E,B(x,r)) \geq \eta$, that is 
\begin{align} \label{eq:densityhypo}
\min\{\mu(E \cap B(x,r)), \mu(B(x,r) \setminus E)\} \geq \eta \mu(B(x,r)).
\end{align}
%$\mu(E \cap B(x,r)) \geq \eta \mu(B(x,r))$ and $\mu(B(x,r) \setminus E) \geq \eta \mu(B(x,r))$,
}%so
By our choices of $\sigma$ and $\delta$, %and by Lemma \ref{lem:neighborhood}, 
%we have
%Scratchwork for $i=1$:\ 
%Based on our choice of $\delta$, with $b > 0$, we simplify
\begin{align*}
C_N(2\delta)^b < 
C_N\cdot \frac{\eta}{2 C_N} \leq
\frac{1}{2}\Theta_\mu(E,B(x,r)),
\end{align*}
so for the case $i=1$, 
using $R=\Omega$ in Lemma \ref{lem:neighborhood} we have
\begin{align*}
%\mu(B \cap \Omega \cap N_\delta) \leq&~
%C_n\delta^b \mu(B\cap \Omega)
%\\
%\mu(F_1 \cap B) = 
\mu(E \setminus N_{2\delta r}(\Omega^c) \cap B(x,r)) \geq&~ 
\mu(E \cap B(x,r)) - \mu(\Omega \cap N_{\delta r}(\Omega^c) \cap B(x,r)) \\ \geq&~
\mu(E \cap B(x,r)) - C_N\delta^b \mu(B(x,r)\cap \Omega) \\ >&~
\Theta_\mu(E,B(x,r)) %\cdot 
\mu(B(x,r)) - \frac{\Theta_\mu(E,B(x,r))}{2} \mu(B(x,r)).
\end{align*}
From this (and replacing $E$ with its complement, for $i=2$) we conclude
\begin{equation}\label{eq:lowerboundF}
\mu(F_i \cap B(x,r)) \geq \frac{\Theta_\mu(E,B(x,r))}{2} \mu(B(x,r)).
\end{equation}
%and
%\begin{equation}\label{eq:lowerboundB}
%\mu(B(x,2r) \setminus N_{\sigma r}(\Omega^c)) \geq 1/2 \mu(B(x,2r)).
%\end{equation}
Now consider the sets
\begin{align*}
\cD :=&~ \{(y,z) \in X \times X\colon y \in B(x,2r), d(y,\Omega^c) \geq \delta r, %4A s r, 
d(z,y) \leq s r\},
\\
\cD'_i :=&~ \{(y,z) \in X \times X\colon z \in B(x,r) \cap F_i, d(z,y) \leq s r\},
\end{align*}
where, clearly, $\cD'_i \subset \cD$. 
Using Fubini's Theorem, 
for $i=1,2$ 
it follows that
\begin{eqnarray*}
\int_{B(x,2r) \setminus N_{\delta r}(\Omega^c)} \frac{\mu(F_i \cap B(y,s r))}{\mu(B(y,s r))} ~d\mu(y) &=&  
%\int\int
\iint_{\cD} \frac{1_{F_i}(z)}{\mu(B(y,s r))} ~d\mu(z) ~d\mu(y) \\ 
&\geq& \frac{1}{D}
%\int
\iint_{\cD'_i} 
\frac{1_{F_i}(z)}{\mu(B(z,s r))} ~d\mu(z) ~d\mu(y) \\
&\geq&\int_{{F_i} \cap B(x,r)}\frac{1}{D\mu(B(z,s r))}\int_{B(z,s r)} 1~d\mu(y) ~d\mu(z) \\
&\geq& 
\frac{1}{D} \mu(F_i \cap B(x,r)).
\end{eqnarray*}
As a result,
there must thus exist some $y \in B(x,2r) \setminus N_{\delta r}(\Omega^c)$ so that
$$
\mu(B(x,2r) \setminus N_{\delta r}(\Omega^c)) \frac{\mu(F_i \cap B(y,s r))}{\mu(B(y,s r))} \geq \frac{1}{D} \mu(F_i \cap B(x,r)).
$$
By Equation \eqref{eq:lowerboundF} we have for such $y\in B(x,2r) \setminus N_{\delta r}(\Omega^c)$ that
$$
\frac{\mu(F_i \cap B(y,s r))}{\mu(B(y,s r))} \geq \Theta_\mu(E,B(x,r)) \frac{1}{2D} \frac{\mu(B(x,r))}{\mu(B(x,2r))} \geq \frac{1}{2D^2}\Theta_\mu(E,B(x,r)).
$$
%Since we applied this to $F=E \setminus N_{\delta r}(\Omega^c)$ and $F=B(x,r) \setminus (E \cup N_{\delta r}(\Omega^c))$, we get two points $y_1,y_2 
Based on the previous choices for $F_i$, there must exist $y_i 
\in B(x,2r) \setminus N_{\delta r}(\Omega^c)$ so that 
\begin{equation} \label{eq:lowerbound_y}
\mu(F_i \cap B(y_i,s r)) \geq \frac{1}{2D^2} \Theta_\mu(E,B(x,r)) \mu(B(y_i,s r)).
\end{equation}
%$$\mu(E \cap B(y_i,s r)) \geq \frac{1}{2D^2} \Theta_\mu(E,B(x,r)) \mu(B(y_i,s r)).$$
%and
%$$\mu(B(y_2,s r) \setminus E) \geq \frac{1}{2D^2} \Theta_\mu(E,B(x,r)) \mu(B(y_2,s r)).$$
%
%If $\mu(B(y_1, s r) \setminus E) \geq \frac{1}{2D^2} \Theta_\mu(E,B(x,r)) \mu(B(y_1, s r))$ or if $$\mu(B(y_2, s r) \cap E) \geq \frac{1}{2D^2} \Theta_\mu(E,B(x,r)) \mu(B(y_2, s r)),$$ then we could use, $y=y_1$ or $y_2$ respectively and get the statement. Note 
%
%Thus, we are left to consider the case where
%
%If we also have 
%$
%\mu(B(y_1,sr)\setminus E) %/ \mu(B(y,sr)) 
%\geq \frac{1}{2D^2}
%\mu(B(y_1,sr))
%$, then the claim would follow for $y \defeq y_1$ since $d(x,y) \leq 2r \leq 2Ar$. Similarly, if 
%$
%\mu(B(y_2,sr)\cap E) %/ \mu(B(y_1,sr)) 
%\geq \frac{1}{2D^2} \mu(B(y_2,sr))
%$, then we could choose $y \defeq y_2$ for the claim.
We can therefore assume
for $i=1,2$ that
\begin{equation}\label{eq:densityest}
\mu(B(y_i,sr)\setminus E) \leq
\frac{1}{2D^2} \mu(B(y_i,sr)). 
%\hspace{.5cm} \text{and} \hspace{.5cm} 
%\mu(B(y_2,sr)\cap E) &\leq& \frac{1}{2D^2} \mu(B(y_2,sr)).
\end{equation}
Indeed, if \eqref{eq:densityest} failed for $i=1$ then since $d(x,y) \leq 2r \leq 2Ar$, the claim would follow for $y \defeq y_1$.  Similarly, if \eqref{eq:densityest} failed for $i=2$, then the claim would hold for $y \defeq y_2$ instead.

Let $\gamma\colon I=[0,1] \to X$ be a $A$-uniform curve joining $y_1$ and $y_2$, %. Then 
so %by definition we have 
%$\diam(\gamma) \leq Ad(y_1,y_2) \leq 2Ar$ and 
by Lemma \ref{lem:basicuniflemma} we get
$$d(\gamma(t),\Omega^c) \geq \frac{1}{4A} \min\{d(y_1,\Omega^c), d(y_2,\Omega^c)\} \geq s r$$
for all $t \in I$. 
Since $X$ is geodesic and doubling, the map 
%$t \to 
$$
T(t) \defeq 
\frac{\mu(B(\gamma(t), s r) \cap E)}{\mu(B(\gamma(t), s r))}
$$
is continuous. From Equation \eqref{eq:densityest} we get $T(0)\geq (1-\frac{1}{2D^2})\geq \frac{1}{2}$ and $T(1)\leq \frac{1}{2D^2}\leq \frac{1}{2}$, so by continuity, 
there is some $t$ so that 
$T(t) = 
%$\frac{\mu(B(\gamma(t), s r) \cap E)}{\mu(B(\gamma(t), s r))} =
\frac{1}{2}$. %, 
%and we can set $y = \gamma(t)$. Clearly 
Moreover,
$$
d(\gamma(t)%y
,x) \leq \diam(\gamma) + d(y_1, x) \leq 
Ad(y_1,y_2) + r \leq 
2Ar + r \leq 
4A r
$$
follows from the definition of an $A$-uniform curve, in which case
$y = \gamma(t)$ satisfies $y \in B(x,4Ar)$ as well as
\begin{align*}
\Theta_\mu(E,B(y,s r)) = 
T(t) = 
\frac{1}{2} \geq
\frac{1}{2D^2} \geq \frac{1}{2D^2} \Theta_\mu(E,B(x,r))
\end{align*}
which is the desired conclusion.
%$d(y,x) \leq \diam(\gamma) + d(y_1, x) \leq 2Ar + r \leq 4A r$, and $\Theta_\mu(E,B(y,s r)) = \frac{1}{2} \geq \frac{1}{2D^2} \geq \frac{1}{2D^2} \Theta_\mu(E,B(x,r))$. 
\end{proof}

By a similar argument one also gets the following.

%{\color{red} TODO:\ rename this as a corollary?}

\begin{corollary}
%\begin{lemma}
\label{cor:smallerscale} %\label{lem:smallerscale} 
Fix $D \geq 1$. Suppose $X \subset \R^d$ is connected, suppose $\mu=\lambda|_X$ is $D$-doubling, and fix $r \in (0,\diam(X))$. There is a constant $L=L(D)$ so that the following holds:\  %$r_0<r<\diam(X)$. 
if $\eta \in (0,1)$, $E \subset X$, and $B(x,r) \subset X$ satisfy %are such that
$$
\Theta_\mu(E,B(x,r)) \geq \eta
$$
then 
for each $
%r_0 
r_1 \in (0,r)$
there exists %both %a constant 
%$L=L(D) 
%> 0
%$ 
%and %a point 
$
x_1 %y
\in X$ so that
$$
\Theta_\mu(E,B(x_1,r_1)) \geq \frac{1}{L}\eta.
$$
%\end{lemma}
\end{corollary}

%[{\color{blue}
%\tt FOLLOW-UP IDEA:\ How about $x_1$ instead of $y$?  In the proof of Theorem 1.3, we already call it $x_1$.%
%}]

\begin{remark}%{proof} 
The proof of Corollary \ref{cor:smallerscale}  proceeds exactly as in %the previous 
Lemma \ref{lem:containedballs}, %above, 
%except %we define $F=E$ and $F=B(x,r) \setminus E$, and for
but with
the following substitutions:\

\begin{itemize}
\item
%instead of using $s$ to specify the scale we just use $B(y,r_1)$;
replace every instance of $sr$ by $r_1$; 
\item replace $N_{\delta r}(\Omega^c)$ and $N_{2\delta r}(\Omega^c)$ by empty sets, and remove where appropriate;
\item
use the entire space $X$, so % and 
no complementary set $\Omega^c$ is needed.  Instead, choose 
\begin{align*}
F_i =&~
\begin{cases}
E, & \text{if } i=1
\\
X\setminus E, & \text{if } i=2
\end{cases}
\\
\mathcal{D} =&~ \{(y,z) \in X \times X \colon d(x,y)<2r, d(z,y) \leq r_1 \}
\\
\mathcal{D}_i' =&~
\{(y,z) \in X \times X \colon z \in B(x,r) \cap F_i, d(z,y) \leq r_1 \}; 
\end{align*}
%$\mathcal{D}=\{(y,z) \in X \times X : d(z,y) \leq r_0 \}$ and $\mathcal{D}'=\{(y,z) \in X \times X : d(z,y) \leq r_0, z \in B(x,r) \cap F \}$. 
%We do not need to define $\sigma$. We get a
\item
%Moreover, 
%at the end of the proof, instead of using the curve $\gamma$, we use the fact 
to finish the proof, use directly 
that $X$ is connected and that the map
$$
z \mapsto \frac{\mu(B(z,r_1) \cap E)}{\mu(B(z,r_1))}
$$
is continuous, in which case no explicit curve $\gamma$ is needed.
\end{itemize}
\end{remark}%{proof}

We need also a lemma on volumes
for subsets of interest in Euclidean spaces.

\begin{lemma} \label{lem:volobstacles} %Assume $X$ is $C_{AR}$-Ahlfors regular.
Under the hypotheses of Theorem \ref{thm:mainthm}, then $S_{\nseq} \subset \R^d$ is $d$-Ahlfors regular with constant $C_{AR}$ %= C_{AR}(A,L,\delta, \epsilon,d,\nseq)$.
depending only on those constants from Definition \ref{def:sparsecoll} and the sequence $\nseq$.
\end{lemma}

\begin{proof}
Let $A,\delta, L$ be the constants for the uniform sparseness condition.
As $\mu$ is the restriction of Lebesgue measure, we clearly have
%Clearly 
$\mu(B_{S_\nseq}(x,r)) \leq \omega_d r^d$, %since we are using the restricted measure. 
%It suffices thus to 
so it suffices to 
show the lower bound
in \eqref{eq:arreg}.
Since $n_k \in \N$ with $n_k \geq 3$ it clearly holds that
%We have 
$$
\sum_{k=1}^\infty \frac{1}{n_k^d} < \sum_{k=1}^\infty \frac{1}{n_k^{d-1}}<\infty.
$$
If $Y = \Omega$, %. %Then, 
then Lemma \ref{lem:arreg} further implies $Y$ is $d$-Ahlfors regular with some constant $C(A,d)$. %=D(A,d)$.  
The same is true with a different constant, if $Y=\Omega \setminus R$ for any one $R \in \mathcal{R}_{\nseq,k}$, as Theorem \ref{thm:cutout} implies that $Y$ is $A'$-uniform, for some $A'$ depending solely on $A$ (and the doubling constant of $\R^d$). 
Either way, $Y$ is Ahlfors $d$-regular with a constant $C=C(A',D)$.

In the following, we prove a lower bound only for some small scales depending on the sequence $\nseq$.
%{\color{red}
%so the implied constant in the Ahlfors regularity will also depend on such a sequence.
%}
%Let $K$ be chosen so that 
In particular, choosing $K \in \N$ so that 
$$%\begin{align} \label{eq:tailseries}
\sum_{k=K}^\infty \frac{1}{n_k^d} \leq \frac{\delta^{d}}{4^{d+1} C 
L^{d}
},
$$%\end{align}
it suffices to prove a lower bound %in estimate \eqref{eq:arreg} 
for $r \in (0, \delta s_{K-1}/2)$
only.  

To this end, choose $J \geq K$ so that 
%$r \in (\delta s_{J}/2, \delta s_{J-1}/2)$.
$\frac{1}{2}\delta s_{J} < r < \frac{1}{2}\delta s_{J-1}$.
Let $x \in S_{\nseq}$. By condition (5) in Definition \ref{def:sparsecoll}, the ball $B(x,r)$ intersects at most one $R \in \mathcal{R}_{\nseq,k}$ for $k \leq J$. %If there is such an 
For any such
$R$, let $Y=\Omega \setminus R$; otherwise let $Y =\Omega$. %Then 
In either case,
%as before
$Y$ is $C$-Ahlfors regular %by Lemma \ref{lem:arreg} 
and satisfies
%Now, 
$$
B(x,r) \cap S_{\nseq} 
=
B(x,r) \cap Y \setminus \bigg(
\bigcup_{k>J} \bigcup_{R \in \mathcal{R}_{\nseq,k}} R
\bigg).
$$ 
%For the moment, assume that 
{%\color{blue}
Now put $\mathcal{R}_{\nseq,k}(x,r) \defeq \{ R \in \mathcal{R}_{\nseq,k} : R\cap B(x,r) \neq \emptyset \}$.}
As a special case, assuming first
that for all $k>J$ we have
\begin{equation}\label{eq:obstaclebound}
\mu\bigg(\bigcup_{R \in \mathcal{R}_{\nseq,k}} B(x,r) \cap R \bigg) \leq 
\Big(
\frac{4 L r}{\delta n_k}
\Big)^d,
\end{equation}
then summing over $k$ gives
$$
\mu\bigg(\bigcup_{k>J}\bigcup_{R \in \mathcal{R}_{\nseq,k}} B(x,r) \cap R \bigg) \leq 
\sum_{k=J}^\infty \mu\bigg(\bigcup_{R \in \mathcal{R}_{\nseq,k}} B(x,r) \cap R \bigg) \leq 
\bigg(\sum_{k=K}^\infty \frac{1}{n_k^d} \bigg)
\frac{4^d L^{d} 
r^d}{\delta^{d}}
\leq
\frac{r^d}{4C}
$$
and along with the previous equality of sets, the desired lower bound follows:
\begin{eqnarray*}
\mu(B(x,r) \cap S_{\nseq}) &=& 
\mu\bigg(B(x,r) \cap Y \setminus \bigcup_{k>J}\bigcup_{R \in \mathcal{R}_{\nseq,k}} R\bigg) \\ &=& 
\mu(B(x,r) \cap Y) - \mu\bigg(\bigcup_{k>J}\bigcup_{R \in \mathcal{R}_{\nseq,k}} B(x,r) \cap R \bigg) \geq \frac{r^d}{2C}.
\end{eqnarray*}
To %obtain 
see how 
estimate \eqref{eq:obstaclebound}
is valid in the general case,
observe that each $R \in \mathcal{R}_{\nseq,k} $ that intersects $B(x,r)$ can also be included in a ball $B(x_R, Ls_k)$ and that the rescaled balls $B(x_R,\delta s_{k-1}/2)$ are disjoint %. All these inflations 
and 
are contained in $B(x,2 r)$.  Thus
%the number of such balls is at most 
there are at most 
$(
\frac{4 r}%{\color{red}Lr}
{\delta s_{k-1}})^d$
%$(4^d r^d ) / (\delta^{d} s_{k-1}^d)$
%Then
such balls
and summing over all previous such $R$ yields
\begin{eqnarray*}
\mu\bigg(\bigcup_{R \in \mathcal{R}_{\nseq,k}} B(x,r) \cap R \bigg) \leq 
\sum_{%\color{red}\stackrel{R \in \mathcal{R}_{\nseq,k}}{
%R\cap B(x,r) \neq \emptyset}
R \in \mathcal{R}_{\nseq,k}(x,r)
} \omega_d L^d s_k^d \leq
\omega_d L^d s_k^d  \frac{4^d r^d }{\delta^{d} s_{k-1}^d} =
\frac{\omega_d 4^d L^{d} r^d}{\delta^{d} n_k^d},
\end{eqnarray*}
which is \eqref{eq:obstaclebound} as desired.
\end{proof}

\section{Isoperimetry for Sierpi{\'n}ski sponges}\label{sec:isoperim}
In this section, 
we develop tools regarding isoperimetry, which are used to prove Theorem \ref{thm:mainthm}.  For completeness, we define Poincar\'e inequalities here, although, we shall quickly pivot to the equivalent notion of (relative) isoperimetry, which is what we actually prove and use.

\begin{definition} \label{def:PI}
Let %$r_0>0$ and 
$1 \leq p  < \infty$. A 
%proper metric measure space $(X,d,\mu)$ with a Radon measure $\mu$ 
closed subset $X$ of $\R^d$, with 
$\mu = \lambda|_X$, %$\mu(B(x,r)){\color{blue}=\lambda|_X(B(x,r))}$ 
% \in (0,\infty)$ for all balls,
is said to satisfy a %\sout{local}
{\sc $p$-Poincar\'e inequality 
}
(with constants $C,\Lambda \geq 1$) %for $1\leq p < \infty$,
 if for all Lipschitz functions $f\colon\ X \to \mathbb{R}$ and all $x \in X$ and $r \in (0,{\rm diam}(X))$ we have for balls $B \defeq B(x,r)$ that
\begin{equation}\label{eq:defPI}
\fint_B |f-f_B| ~d\mu \leq C r  \left(\fint_{\Lambda B} 
|\nabla f|^p%\Lip [f]^p 
~d\mu \right)^{1/p}.
\end{equation}
%{\color{red} If $r_0 = \infty$, then say that $X$ satisfies a {\sc (global) $(1,p)$-Poincar\'e inequality} (with the same constants).}
\end{definition}

Here, for any %measurable and 
locally Lebesgue integrable $f \co X \to \R$ its average value on a ball is
$$
f_B \defeq \fint_B f ~ d\mu \defeq \frac{1}{\mu(B)} \int f ~d\mu,$$
and 
by Rademacher's Theorem, the restriction of the gradient $\nabla f$ to $X$ is well-defined almost-everywhere.

This inequality is essentially a local property, as the following version of \cite[Theorem 1.3]{bjornlocal} shows. This quantitative version does not appear explicitly in the reference, but follows directly from their argument.

\begin{lemma}\label{lem:local}Suppose that $(X,d,\mu)$ is a 
 bounded,
connected, $D$-doubling metric measure space. %with finite diameter $\diam(X) \in (0,\infty)$. 
If $X$ satisfies  Definition \ref{def:PI} with constant $(C,\Lambda)$ for all $r \in (0,r_0)$, then $X$ satisfies Definition \ref{def:PI} for all $r>0$ with constants $C_1=C_1(C,\Lambda, \frac{\diam(X)}{r_0})$ and $\Lambda_1=\Lambda_1(C,\Lambda,\frac{\diam(X)}{r_0})$.

\end{lemma}

\subsection{Definition of isoperimetry and iteration}

Lahti and Korte discuss in \cite{lahtikorteisoperim} various criteria that are equivalent %versions of this 
to a $1$-Poincar\'e inequality. %In particular, they show the following.
To formulate them, 
we will require two additional notions.
If $\mu$ is a doubling measure on $X$, then put %we define
$$
h(B(x,r)) \defeq \frac{1}{r}\mu(B(x,r))
$$
and define for any set $E \subset X$ the {\sc codimension-one Hausdorff content} as
$$
\Ha_{h,\delta}(E) \defeq \inf \bigg\{\sum_{i=1}^\infty h(B(x_i,r_i)) \ \  \bigg| \ \  r_i \leq \delta, E \subset \bigcup_{i=1}^\infty B(x_i,r_i)\bigg\},
$$
and the {\sc codimension-one Hausdorff measure} as
$$
\Ha_{h}(E) \defeq \lim_{\delta \to 0} \Ha_{h,\delta}(E).
$$

If we use $h(B(x,r))=r^{s}$ instead, we obtain the $s$-dimensional (spherical) Hausdorff measure $\Ha_s(E)$ (up to a scalar multiple, depending on convention).

\begin{remark}\label{rmk:comparabiliy}
Note that if $\mu$ is Ahlfors $Q$-regular with constant $C_{AR}$, then $\Ha_h$ is comparable to %the 
$(Q-1)$-dimensional (spherical) 
Hausdorff measure $\Ha_{Q-1}$. Indeed, we have
$$
\frac{1}{C_{AR}} \Ha_{Q-1}(E) \leq \Ha_h(E) \leq C_{AR} \Ha_{Q-1}(E),
$$
as follows easily from the definition. We will use this fact, as we will be giving bounds for the $d-1$-dimensional Hausdorff measure instead of the codimension-one Hausdorff measure. 
\end{remark}

We are now ready to introduce the criterion \cite[Theorem 1.1]{lahtikorteisoperim}.

\begin{theorem}[Korte-Lahti] \label{thm:11PI} A doubling metric measure space $(X,d,\mu)$ satisfies the $1$-Poincar\'e inequality if and only if %it satisfies the  {\sc relative isoperimetric inequality}, that is, if 
there are constants $C_S, \Lambda_S \in [1,\infty)$ such that for every ball $B=B(x,r)$ and any Borel set $E \subset X$ we have
\begin{equation} \label{eq:isoperim}
\Theta_\mu(E,B) \leq
%\frac{\min\{\mu(E\cap B), \mu(B \setminus E) \}}{\mu(B)} \leq 
C_{S} r \frac{\Ha_{h}(\partial E \cap \Lambda_S B)}{\mu(\Lambda_S B)}.
\end{equation}
\end{theorem}

%\begin{remark}
Inequality \eqref{eq:isoperim} is known as a  {\sc relative isoperimetric inequality}:\ once a ball is given, the measure of %it defines a relationship between 
the boundary of %the set 
$E$ %in question and 
within that ball controls the density of $E$ and its complement relative to that ball.
To specify the dependence on parameters, we sometimes refer to \eqref{eq:isoperim} as a {\sc relative isoperimetric inequality with constants $C_S$ and $\Lambda_S$.}
%\end{remark}

%\begin{note}[\underline{Jasun, 2019-Jan-29}]
%{\edit Combine Definition \ref{def:isoperim} and Theorem \ref{thm:11PI}?}
%\end{note}

The relative isoperimetric inequality
can be considered as a property of every subset $E$ in $X$, %of
relative to
every ball in the space $X$. However, it will be helpful to introduce a ``density level'' in our proofs, i.e.\ to consider isoperimetric inequalities only for ``large enough'' sets.

\begin{definition} \label{defn:relisoconsts}
Let $\tau, C > 0$. A metric measure space $(X,d,\mu)$ is said to satisfy a {\sc $(\tau, C)$-isoperimetric inequality} if there is a $\Lambda \geq 1$ such that every Borel set $E \subset X$ and ball $B=B(x,r)$ satisfies the following:\
if $\Theta_\mu(E,B) \geq \tau$, then 
\begin{equation}\label{eq:def-atlevel}
\Theta_\mu(E,B) \leq Cr\frac{\Ha_{h}(\partial E \cap \Lambda B\big)}{\mu(\Lambda B)}.
\end{equation}

To specify the dependence on parameters, we say that $X$ satisfies a {\sc $(\tau, C)$-isoperimetric inequality with inflation factor $\Lambda$}.  
\end{definition}

Since the left hand side of %estimate
\eqref{eq:def-atlevel} is bounded below by $\tau$, it would really suffice to give just a constant lower bound. However, we wish the constants to match those in Equation \eqref{eq:isoperim} as closely as possible, and simply to weaken the condition by restricting the sets considered.

%{\color{blue}Take away? In the case of collections of balls on doubling metric measure spaces, Definition \ref{defn:relisoconsts} is equivalent to a relative isoperimetric inequality, where the constants are quantitative.}

\begin{proof}[Proof of Theorem \ref{thm:largeenough}]
Given parameters $D \geq 2$, $\tau \in (0,D^{-3}]$, $C > 0$, and $\Lambda \geq 1$,
we will assume 
that $X$ satisfies 
 the $(C,\tau)$-isoperimetric inequality with inflation factor $\Lambda$ and prove the isoperimetric inequality \eqref{eq:isoperim} with $C_S = %C_S(D,\Lambda)=
D^{7+\log_2(\Lambda)}C$
and
$\Lambda_S=%\Lambda_S(\Lambda)= 
2\Lambda$. 

Let $E$ be any Borel subset of $X$. 
Without loss of generality assume 
$$
\mu(E\cap B(x,r)) < \frac{1}{2}\mu(B(x,r)),
$$
otherwise we prove the inequality by replacing $E$ by  $E^c$. 
Now if 
%{\color{red} This should be with 2r, because we can't estimate $\mu(E \cap B(x,2r)) \leq D \mu(E \cap B(x,r))$.}
$x \in X$ and $r > 0$ satisfy
$$
\frac{\mu(E\cap B(x,2r))}{\mu(B(x,2r))} \geq \frac{1}{D^3}
$$
then 
by hypothesis, the inequality equation \eqref{eq:def-atlevel} is exactly what we want for \eqref{eq:isoperim} except for an extra factor of $D$ arising when the quantity $\mu(\Lambda B)$ is adjusted for $\mu(2\Lambda B)$.

We can therefore assume that
$$
%\frac{\mu(E\cap B(x,r))}{\mu(B(x,r))} < \frac{1}{D^3}
%\text{ and hence }
\frac{\mu(E\cap B(x,2r))}{\mu(B(x,2r))} < %\leq 
%D\frac{\mu(E\cap B(x,r))}{\mu(B(x,r))} < 
\frac{1}{D^3}.
$$
Consider the set of density points
 $$
 S = \Big\{
 z \in B(x,r) ~\Big|~ \lim_{t \to 0} \frac{\mu(B(z,t) \cap E)}{\mu(B(z,t))}  = 1
 \Big\}.
 $$
 By Lebesgue differentiation, we have $\mu(S) = \mu(E \cap B(x,r))$ and for each $z \in S$ we have
%$$
%\lim_{t\to 0}\frac{\mu(E \cap B(\textcolor{red}{z},t))}{\mu(B(\textcolor{red}{z},t))}  = 1
%$$
%and
$$
\frac{\mu(E \cap B(z,r))}{\mu(B(z,r))} \leq 
D^2\frac{\mu(E \cap B(x,2r))}{\mu(B(z,4r))} \leq 
D^2\frac{\mu(E \cap B(x,2r))}{\mu(B(x,2r))} < 
\frac{1}{D}.
$$
%$$
%\frac{\mu(E \cap B(z,\textcolor{red}{\Lambda_\mathcal{B} }r))}{\mu(B(z,\textcolor{red}{\Lambda_\mathcal{B} }r))} \leq 
%D^2\frac{\mu(E \cap B(x,2\textcolor{red}{\Lambda_\mathcal{B} }r))}{\mu(B(z,4\textcolor{red}{\Lambda_\mathcal{B} }r))} \leq 
%D^2\frac{\mu(E \cap B(x,2\textcolor{red}{\Lambda_\mathcal{B} }r))}{\mu(B(x,2\textcolor{red}{\Lambda_\mathcal{B} }r))} \leq 
%\frac{1}{D}.
%$$
 Thus, from Lemma \ref{lem:goodscale} there exists $r_z \leq r$ with
\begin{equation}\label{eq:densityestz}
    \frac{1}{D^2} <
    \frac{\mu(B(z,r_z) \cap E)}{\mu(B(z,r_z))} < %\leq
    \frac{1}{D}.
\end{equation}
By the $5B$-covering lemma (see e.g. \cite[Theorem 2.1]{Mattila1999}), there is a countable subset $\{z_i\}_{i\in I}$ of $S$ %for $i \in I$, 
 and radii $s_i = r_{z_i}$ such that $\{B(z_i,\Lambda s_i)\}_{i\in I}$ is pairwise disjoint, that $B_i \defeq B(z_i,s_i)$ satisfy \eqref{eq:densityestz}, and that
$$
S \subset \bigcup_{i \in I} B(z_i, 5\Lambda s_i).
$$
 By \eqref{eq:densityestz} %\eqref{eq:isoperim} 
 and the hypotheses of the theorem, each %of the sets 
 $E \cap B_i$ %within the balls $B_i$ 
 satisfies %the lower density assumption 
 $\Theta_\mu(E,B_i) \geq \frac{1}{D^2} > \tau$ and hence the $(\tau,C)$-isoperimetric inequality as well:
 \begin{equation}\label{eq:lowerboundary}
 \Ha_{h}(\partial E \cap  B(z_i,\Lambda s_i)) \geq \frac{\mu(B(z_i, \Lambda s_i))}{C s_i D^2} \geq \frac{\mu(B(z_i, 5\Lambda s_i))}{C s_i D^5}.
\end{equation}
From this and the inclusion
$B(z_i,\Lambda s_i) \subseteq B(x,2\Lambda r)$, for each $i\in I$, 
it follows that
\begin{eqnarray*}
    \Ha_{h}(\partial E \cap B(x,2 \Lambda r)) &\geq&
    \sum_{i \in I} \Ha_{h}(\partial E \cap  B(z_i, \Lambda s_i))
    \\ 
    &\stackrel{\eqref{eq:lowerboundary}}{\geq}& \frac{1}{C D^{5}}
		\sum_{i \in I} \frac{\mu(B(z_i, 5\Lambda s_i))}{s_i} \\
		&\geq& \frac{1}{CD^{5}}
		\frac{\mu(S)}{r} \\
		&\geq& \frac{1}{2CD^{7+\log_2(\Lambda)}} 
		\frac{\mu(B(x,2\Lambda r))}{r}
		\frac{\mu(E \cap B(x,r))}{ \mu(B(x,r))} \\
		&\geq& \frac{1}{C_S} 
		\frac{\mu(B(x,2\Lambda r))}{r} \Theta_\mu(E,B(x,r)) %\\
		%&~\geq~& \frac{1}{D^{4 + \log_2(\lambda_B)}C_{\mathcal{B}}}\frac{\min\{ \mu(E \cap B(x,r)), \mu(B(x,r) \setminus E)\}}{r}.
 \end{eqnarray*}
which implies the relative isoperimetric inequality, as desired.
\end{proof}

\subsection{A Euclidean isoperimetric inequality}

%Towards the proof of the main result, 
%we now proceed with a few technical lemmas that are distinctly Euclidean.

We will prove a ``projected'' isoperimetric inequality for 
Borel sets $E$ relative to 
rectangles in Euclidean spaces $\R^d$. This 
%allows us to show 
guarantees 
that $\partial_{S_\nseq} E$ has large projections, and when combined with condition (6) from Definition \ref{def:smallproj}, gives lower bounds for $\mathcal{H}_{d-1}(\partial_{S_{\nseq}} E)$. %To give lower bounds for the size of $\partial_{S_\nseq} E$ we will consider a directed version of the
To this end 
we formulate an isoperimetric inequality in terms of a direction-wise
Euclidean boundary. %$\partial_{+,i} E$. 

For $x \in \R^d$, denote by $l_{i,x}$ the line containing $x$ that is parallel 
to the $i$'th coordinate axis.
If $E \subset \R^d$, we also put
$$
\partial_{+,i} E = 
\{\, x \, |\,  x \in \partial_{l_{i,x}}(l_{i,x} \cap E)\, \}.
$$
In other words, the set $\partial_{+,i} E$ consists of those points $x$, where a sequence of points 
exists in the $i$'th coordinate direction within $E$, and outside of $E$, which converges to it. Next, denote by $\pi_i$ the projection of $\R^d$ onto the hyperplane defined by $x_i = 0$.  

We now relate the density of sets with respect to boxes with the size of the projections of their boundaries.  The following lemma is likely classical, but we include its proof for completeness.

\begin{lemma}\label{lem:projection} 
Let $Q= \prod_{i=1}^d (a_i, b_i)$ be a rectangle, and $E \subset \R^d$ a Borel set. Then,
$$
\Theta_\lambda(E,Q) \leq 
d\sum_{i=1}^d \frac{\Ha_{d-1}(\pi_i(\partial_{+,i} E \cap Q))}{\Ha_{d-1}(\pi_i(Q))}
$$

\end{lemma}

\begin{proof}
The statement is invariant under affine functions of $x_i$, so
we can assume  that $Q = (0,1)^d$. Also, the statement is clear for $d=1$.

Towards a proof by induction, assume that the statement has been proven for dimension $d-1$ and that $d\geq 2$.  Without loss of generality assume
 $$
 \lambda_E=\lambda(E \cap Q)% (0,1)^d) 
\leq \frac{1}{2}.
 $$
For $t \in (0,1)$ and $\mathbf{y} \in (0,1)^{d-1}$ define \begin{eqnarray*}
 H_t &\defeq& \{(t,\mathbf{x}) \in (0,1)^{d-1} ~|~ \mathbf{x} \in (0,1)^{d-1}\}
 \\
 l_{\mathbf{y}} &\defeq& \{(s,\mathbf{y}) \in (0,1)^{d-1} ~|~ s \in (0,1)\}
 \end{eqnarray*}
and consider the following images under $\pi_1$:
 \begin{eqnarray*}
 I_0 &
\defeq& \{\mathbf{x} \in (0,1)^{d-1} ~|~ l_{\mathbf{x}} \subset E \cap Q\},
 \\
 O_0 &\defeq& \{\mathbf{x} \in (0,1)^{d-1} ~|~ l_{\mathbf{x}} \cap E \cap Q = \emptyset\},
 \\
 P_0 &\defeq& (0,1)^{d-1} \setminus (I_0 \cup O_0),
\end{eqnarray*}
in which case it is clear that
\begin{eqnarray*}
 P_0 &=& \pi_1(\partial_{+,1} E \cap Q)
 \\
 \pi_1(E \cap Q) &\subset& P_0 \cup I_0
 \\
 \pi_1(Q \setminus E) &\subset& P_0 \cup O_0.
\end{eqnarray*}
If $\Ha_{d-1}(P_0) \geq \frac{1}{d}\lambda_E$ then the statement of the lemma is trivially true, so assume $\Ha_{d-1}(P_0) \leq \frac{1}{d}\lambda_E$ 
which implies that
$$
\lambda_E \leq
\Ha_{d-1}(\pi_1(E \cap Q)) \leq
\Ha_{d-1}(P_0 \cup I_0) \leq
\frac{1}{d}\lambda_E + \Ha_{d-1}(I_0).
$$
Then, $\Ha_{d-1}(I_0) \geq \frac{d-1}{d}\lambda_E$. Similarly, we assumed $\lambda_E \leq \frac{1}{2} \leq \lambda(Q \setminus E)$ so it follows that
%$
%\Ha_{d-1}(O_0) \geq \frac{d-1}{d}\lambda_E
%$
%and that
$$
\min\{ \Ha_{d-1}(I_0) , \Ha_{d-1}(O_0) \} \geq \frac{d-1}{d}\lambda_E.
$$
Now if $t \in (0,1)$ then $E \cap H_t$ contains a translate of $I_0$, and its complement contains a translate of $O_0$. Then, we get from the $d-1$-dimensional statement for $E \cap H_t$, that
 $$
 \frac{(d-1)\lambda_E}{d} \leq (d-1)\sum_{j=2}^d \Ha_{d-2}(\pi_j(\partial_{+,j} E \cap H_t)),
 $$
 which when integrated over $t$ and using Fubini's theorem gives
 $$
 \lambda_E \leq d\sum_{j=2}^d \Ha_{n-1}(\pi_j(\partial_{+,j } E \cap Q)).
 $$
 This gives the desired inequality.
\end{proof}

\begin{remark} The above gives a fairly simple inductive proof of the isoperimetric inequality in $\R^d$, although with sub-optimal constants.
\end{remark}

\section{Proof of Theorem \ref{thm:mainthm}}\label{sec:proof}

To begin, recall from \S2 that subscripts for balls indicate the space (and hence, the choice of metric) on which those balls are defined.

{%\color{blue}
\begin{remark}[Dependence on parameters]
In the %following
proof below,
many choices will depend on 
the
%$D,C_{AR},d,\nseq$ as well as other 
parameters
from Definitions \ref{def:sparsecoll} and \ref{def:smallproj}, as well as lemmas from earlier sections.
%Mostly we will get these dependencies from Lemmas. There are two parameters 
Two parameters, $\beta$ and $\epsilon_1$, %which we will 
will be determined at the end of the proof, as they depend on many intermediate parameters. However, none of the other parameters will depend on the choice of %these.
$\beta$ and $\epsilon_1$.
\end{remark}
}

\begin{proof}[Proof of Theorem \ref{thm:mainthm}]
Definitions \ref{def:sparsecoll} and \ref{def:smallproj} are scale-invariant, so without loss of generality we may 
assume that %rescale $\Omega$ so that 
$\diam(\Omega)=1$. By a coordinate change we can take the maps $\pi_i$ to be orthogonal projections onto hyperplanes given by $x_i=0$, for $i=1,\dots, d$.

Before fixing other parameters, we first claim that $S_{\nseq}$ is connected. Indeed, by Theorem \ref{thm:cutout} the domains obtained by removing finitely many co-uniform sets from $\Omega$, 
{%\color{blue}
i.e.
\begin{align} \label{eq:presponge}
S_{j,\nseq} := 
\Omega \setminus 
\Big(
\bigcup_{k=1}^j \bigcup_{R \in \mathcal{R}_{\nseq,k}} R
\Big)
\end{align}
are a nested sequence of compact connected sets,
%for each $j \in \N$,
%are each uniform and thus connected,
%. The nested intersection of compact connected sets is also connected, and thus $S_{\nseq}$ is connected. 
so the intersection
$S_{\nseq} = \bigcap_{j=1}^\infty S_{j,\nseq}$
is also connected.
}

It follows from Lemma \ref{lem:volobstacles} that $S_{\nseq}$ is $d$-Ahlfors regular with %some 
constant  $C_{AR}\geq \omega_d$ depending on all parameters $A,L,\delta,d,\nseq$.  Thus it is also $D$-%measure 
doubling with 
fixed constant
$D=2^dC_{AR}^d$.%, %depending 
%with dependence
%on the same parameters. %We hold this doubling constant as fixed.

%In the following many choices will depend on $D,C_{AR},d,\nseq$ as well as other parameters. Mostly we will get these dependencies from Lemmas. There are two parameters, $\beta$ and $\epsilon_1$, which we will determine at the end of the proof as they depend on many intermediate parameters. However, none of the other parameters will depend on the choice of these.

We now proceed in four steps, indicating reductions and strategies as needed.
\vspace{.3cm}

\noindent \textbf{Step I:\ Reduction to a small sum:} 
%By Lemma \ref{lem:volobstacles} 
Since
$S_{\nseq}$ is %$C_{AR}$-Ahlfors regular, and thus 
$D$-doubling, %. By 
Theorem \ref{thm:cutout} implies %we have 
that for any $k \in \N$ and for any fixed obstacle $R \in \mathcal{R}_{\nseq, k}$ the set %cut-out domain 
$\Omega \setminus R$ is 
{%\color{red}
$A'$-uniform%
} 
for some $A' \geq A$. 
{%\bf\color{red}
%We claim that it will suffice to prove the claim for an $A'$-uniform domain and with 
%$$
%\sum_{k=1}^\infty \frac{1}{n_k^{d-1}}<\epsilon_1
%$$ 
%for any desired $\epsilon_1 > 0$. Indeed, choose $N$ so that 
%$$
%\sum_{k=N}^\infty \frac{1}{n_k^{d-1}}<\epsilon_1.
%$$
}

%Now for any $N \in \N$, 
{%\bf\color{blue} 
With %$\epsilon_1 > 0$ 
{%\color{blue}\bf
$\epsilon_1 \in (0,\frac{1}{2})$%
}
to be determined later, 
choose $N$ so that 
$$
\sum_{k=N}^\infty \frac{1}{n_k^{d-1}}<\epsilon_1.
$$

Now for the subsequence $\nseq' = (n_{N+i})_{i=1}^\infty$,
%where $S_{\nseq'}$ is constructed analogously as
%
construct $S_{\nseq'}$ analogously as 
$S_{\nseq}$ but with %$n_{i+N-1}$ in place of $n_i$, for each $i \in \N$
$\mathcal{R}_{\nseq,k+N}$ in place of $\mathcal{R}_{\nseq,k}$
for each $k \in \N$, and possibly with $\Omega \setminus \rho$ in place of $\Omega$. 
Clearly %We have 
$S_\nseq \subset S_{\nseq'} \subset \R^d$, %{\bf and since $S_{\nseq}$ is $C$-Ahlfors regular with the restricted measure $S_{\nseq'}$ is also $C_{AR}$-Ahlfors regular for some $C_{AR}$ depending on $C$ and the dimension.}
so the $d$-Ahlfors regularity of $S_\nseq$ implies the $d$-Ahlfors regularity of $S_{\nseq'}$, 
with the same 
constant $C_{AR}$ or smaller, and hence the same fixed doubling constant $D$ as above (or smaller).

So at this level $N$,
}if $r< \delta s_{N-1}/2$ then by condition (5) of Definition \ref{def:sparsecoll} the ball $B(x,r)$ can intersect only one set $\rho$ in the collection $\{\Omega^c\} \cup \bigcup_{k=1}^{N-1} 
\mathcal{R}_{\nseq,k}$.
%{\color{blue}
 so 
%for the subsequence $\nseq' = (n_i)_{i=N}^\infty$ 
it holds that
\begin{align} \label{eq:equalityballs}
%B=
B_{S_{\nseq}}(x,r)=S_{\nseq} \cap B(x,r) = S_{\nseq'} \cap B(x,r)=B_{S_{\nseq'}}(x,r)
\end{align}
%{\bf holds, for some $S_{\nseq'}$ constructed by removing obstacles in $\bigcup_{k=N}^{\infty}\mathcal{R}_{\nseq,k}$ from a $A'$-uniform domain $\Omega'$ which is either $\Omega$ or $\Omega \setminus R$, where $R \in \bigcup_{k=1}^{N-1}  \mathcal{R}_{\nseq,k}$. }
%MOVED DISCUSSION OF S_{\nseq'}
It suffices to prove a Poincar\'e inequality for $S_{\nseq'}$,
say %{\color{blue} say}
with constants $C_{PI},\Lambda_{PI}$.  To see why,
%
%If we consider 
by applying \eqref{eq:equalityballs} 
%for any fixed $C$, $\Lambda$ as in Definition \eqref{def:PI},
and considering only radii $r \in (0,\delta s_{N-1}/(2\Lambda_{PI}))$, 
this gives a local Poincar\'e inequality in $S_{\nseq}$ %for $r\in (0,\delta s_{N-1}/(2\Lambda))$ 
as both sides of inequality \eqref{def:PI} coincide in $S_{\nseq}$ and $S_{\nseq'}$. 
The set $S_{\nseq}$ is bounded, $D$-doubling and connected, and so, by Lemma \ref{lem:local}, %it satisfies 
a local Poincar\'e inequality further yields 
a global Poincar\'e inequality. 

We now simplify notation by only considering $S_{\nseq'}$ and dropping the primes, that is by increasing $A$ we assume now that $A=A'$ %and 
%$$
%\sum_{k=1}^\infty \frac{1}{n_k^{d-1}}<\epsilon_1
%$$
and that $S_{\nseq}=S_{\nseq'}$ is $d$-Ahlfors regular with constant $C_{AR}$. %where we will determine the required choice for $\epsilon_1$ later. We can also deduce a uniform doubling constant, which from now on we denote $D$.
We will also use a simplified notation for balls and for relative boundaries, that is:\
$B_\nseq(x,r) := B_{S_\nseq}(x,r)$
and
$\partial_\nseq E := \partial_{S_\nseq}E$. Further, for each $i \in \N$ we re-index $n_{i+N-1}$ as $n_i$.
%Moreover, if 
If
necessary, we replace $\Omega$ by $\Omega \setminus \rho$ as before. 

%We also remark, that by this restriction, we need not alter the bound for the Ahlfors-regularity or doubling constant, since the dependence in Lemma \ref{lem:volobstacles} of the doubling constant is uniform as long as we choose  subsequences of a given sequence $\nseq$.
\vspace{.3cm}

\noindent \textbf{%\color{blue}
Step II:\ The %case of small sum 
strategy for small sums
$
\sum_{k=1}^\infty \frac{1}{n_k^{d-1}}<\epsilon_1
$: 
%and an $A$-uniform and co-uniform construction, when $S_{\nseq}$ is $C_{AR}$-Ahlfors regular:
} 
%First, we will impose a condition $\epsilon_1 \leq \frac{1}{2}$.
By Theorems \ref{thm:largeenough} 
{%\color{red}
and \ref{thm:11PI}, %
}
it suffices to prove that there are constants
{%\color{red} 
$C,\Lambda$%
}
so that for sets $E \subset S_{\nseq}$ and balls $B_\nseq(x,r)$, 
if
%with 
$x \in \Omega$ %so that if
satisfies
$$\Theta_\mu(E,B_\nseq(x,r)) \geq \frac{1}{D^3}$$
then it would follow that
\begin{equation}\label{eq:goal}
\frac{Cr \mathcal{H}_h(\partial_\nseq E \cap B_\nseq(x,\Lambda r))}{\mu(B_\nseq(x,\Lambda r))} \geq \Theta_\mu(E,B_{\nseq}(x,r)).
\end{equation}
It suffices to take $r \in (0,\diam(S_{\nseq})]$.  We fix such a ball, as well as a set $E$ for the remainder.
By Remark \ref{rmk:comparabiliy} it suffices to prove this estimate with $\mathcal{H}_{d-1}$ replacing $\mathcal{H}_h$, %. The constant of comparability is related to $C_{AR}$, but since our constant will depend on everything in the statement, this is harmless.
in which case the result follows with constant $CC_{AR}$ in place of $C$.

The proof will proceed by finding a ``good'' ball 
%$B_\nseq(y,s)$ 
in two steps.
We construct, as necessary, balls $B_\nseq(x_1,r_1)$ and $B(x_2,r_2)$, with $x_i\in \Omega$, so that $\Theta_\mu(E,B_\nseq(x_1,r_1)) \geq \eta_1$ and
$\Theta_\lambda(E,B(x_2,r_2)) \geq \eta_2$. In order to pass to Euclidean bounds and apply Lemma \ref{lem:projection}, the second ball in this process will be a Euclidean ball.
%{\color{red} See below} 
{%\color{blue}
See the end of Step III below}
for the precise choices of $\eta_i$ which are quantitative in the previous parameters. 
Putting $x=x_0$ and $r=r_0$, 
at each step we will also ensure that 
%$d(x_1,x)\leq S_1r$, 
%$d(x_2,x_1) \leq S_2 r_1$, and 
$d(x_i,x_{i-1}) \leq S_i r_{i-1}$, and 
%$r_1 \in (\frac{1}{S_1} r,r)$,
%$r_2 \in (\frac{1}{S_2} r_1,r_1)$ 
$r_i \in [\frac{1}{S_i} r_{i-1},r_{i-1}]$ 
for some $S_1,S_2 \geq 0$ which 
depend quantitatively on the parameters.
%but which
%we do not make completely explicit.

As a result, we show that $x_2 \in \Omega$ satisfies both $d(x_2,x) \leq Sr$ and $r \geq r_2 \geq \frac{1}{S}r$
%where $y \in S_{\nseq}$, $d(y,x) \leq S r$ and $s \geq \frac{1}{S} r$, 
for some universal 
constant
$S = S_1S_2$,
as well as
\begin{equation}\label{eq:reduction}
\mathcal{H}_{d-1}(\partial_\nseq E \cap 
%B_\nseq(y,\sqrt{d} s)) \geq \Delta s^{d-1}
B(x_2,\sqrt{d} r_2)) \geq \Delta r_2^{d-1}
\end{equation}
for some universal $\Delta$ depending on all of the constants in the statement.

%[{\color{red}\tt TODO:\ change $y$ above to $x_2$ and move $x_i$ discussion up here.}]

Moreover 
by doubling, the fact that $\Theta_\mu(E,B_\nseq(x,r)) \geq \frac{1}{D^3}$, and the choice of the good balls, %as well as the fact that $d(y,x) \leq S r$, and $s \geq \frac{r}{S}$, 
we can deduce estimate \eqref{eq:goal} from proving inequality \eqref{eq:reduction}. The constants $C,\Lambda$ come directly from doubling and Ahlfors regularity.

With the fixed parameters $A$, $D$, $\delta$, and $L$ from before, let $A'=A'(A,D,\delta,L)$ be the uniformity constant from Theorem \ref{thm:cutout}. This coincides with an upper bound for the uniformity constant of $\Omega \setminus R$ for any fixed obstacle $R$. We note for clarity, that this instance of $A'$ is different from the previous $A'$ -- indeed, due to the abbreviation of notation, the new uniformity constant would arise from the removal of up to two obstacles from the domain we started from.

Since $S_{\nseq} \subset \Omega$, and $S_{\nseq} \subset \Omega \setminus R$ for any obstacle $R$, then both $\Omega$ and $\Omega \setminus R$ are Ahlfors regular with constants which are the same or better than $C_{AR}$. Thus, they are also $D$-doubling.

\vspace{.2cm}
%\noindent \textbf{Reduce to the case of a small ball:} 
\noindent \textbf{
%\color{blue} 
Step III:\ Choosing a good ball:} 
%{%\color{blue}
%Towards determining $y$
%%At each step we construct a ball 
%we consider, as necessary, balls%
%}
%$B_\nseq(x_i,r_i)$ for $i=1,2$ so that $\Theta_\mu(E,B_\nseq(x_i,r_i)) \geq \eta_i$ for some $\eta_i$ which is quantitative in the previous parameters. At each step we will also ensure $d(x_1,x)\leq S_1r, d(x_2,x_1) \leq S_2 r_1$, and $r_1 \in (\frac{1}{S_1} r,r), r_2 \in (\frac{1}{S_2} r_1,r_1)$ for some $S_1,S_2 \geq 0$ which we do not make completely explicit.
%
%Either $r<\frac{\delta}{4\sqrt{d}}$, or $r \in (\frac{\delta}{4\sqrt{d}}, \diam(S_{\nseq})].$ In the first case we set $x_1=x$ and $r_1=r$. In the latter case, we apply Corollary \ref{cor:smallerscale} to find a $L$ depending quantitatively on the parameters, a $r_1 = \frac{\delta}{4\sqrt{d}}$ and a ball $B_\nseq(y,r_1)$
 Put $S_1 = \frac{4\sqrt{d}}{\delta}$.
If $r\leq \frac{\delta}{16A'\sqrt{d}}$ then 
%put 
it suffice to choose
$x_1=x$ and $r_1=r$ and $\eta_1 = \frac{1}{D^3}$.
Otherwise $r \in (\frac{\delta}{{16 A'}\sqrt{d}}, \diam(S_{\nseq})]$, so %put
choose first
$r_1 = \frac{\delta}{16 A'\sqrt{d}}$.
By Corollary \ref{cor:smallerscale} there exist $L_1> 0$ and $x_1 := y \in S_\nseq$, depending quantitatively on the parameters, so that %the ball $B_\nseq(x_1,r_1)$ satisfies
\begin{equation}\label{eq:r_1lowerbound}
\Theta_\mu(E,B_\nseq(x_1,r_1)) \geq \frac{1}{L_1 D^3},
\end{equation}
%Set 
in which case choose instead
$\eta_1=(L_1D^3)^{-1}$. % and $x_1=y$. 
Note %that $d(x_1,x)=0$ 
in the first case
that $d(x_1,x)=0$
and in the second case that
$$
d(x_1,x) \leq \diam(S_{\nseq}) \leq \frac{4\sqrt{d}\diam(S_{\nseq})}{\delta} r 
%~=: 
=
S_1r.
$$
In either case, the distance is bounded by %some constant 
$S_1r$ and  we have $S_1^{-1}r\leq r_1\leq r$.
%\vspace{.2cm}
%\noindent \textbf{Avoid complement of $\Omega$ and large obstacles:} 
Now choose $k \in \N$ so that 
$$
\frac{\delta}{
{%\color{blue} 
16A'%
}
\sqrt{d}} s_{k} \leq 
r_1 \leq 
\frac{\delta }{
{%\color{blue}
16A'%
}
\sqrt{d}}s_{k-1}.
$$
If $B(x_1,{
%\color{blue}
8A'%
}
\sqrt{d}r_1)$ does not intersect 
$\Omega^c$ or
any obstacles $R$ in $\mathcal{R}_{\nseq,l}$ for $l \leq k$, 
then choose 
%$(x_2,r_2)=(x_1,r_1)$.
$x_2=x_1$ and $r_2=r_1$.
Moreover, 
with $\beta > 0$ to be determined later, if
instead
$B(x_1,{
%\color{blue}
8A'%
}
\sqrt{d} r_1)$  does not intersect $\Omega^c$ but does intersect an obstacle and if the largest such obstacle $R$ in $\mathcal{R}_{\nseq,l}$ for $l \leq k$ satisfies $\diam(R) \leq \beta r_1$, then we also %define $(x_2,r_2)=(x_1,r_1)$. 
choose $x_2=x_1$ and $r_2=r_1$.
In both cases, putting $\eta_2=\eta_1$  yields 
\begin{equation} \label{eq:goodballdensity}
\Theta_{%\color{red}
\lambda}(E,B(x_2,r_2)) \geq 
%\eta_1.
{%\color{red}
\omega_d^{-1}C_{AR}^{-1}}\eta_1.
\end{equation}
Here, we use \eqref{eq:r_1lowerbound} and
 the Ahlfors regularity of $S_{\nseq}$, i.e.
that $\mu(B_\nseq(x_2,r_2))\geq 
\frac{1}{\omega_d C_{AR}} \lambda(B(x_2,r_2))$.

If neither of these two cases occurs, then %either 
$B(x_1,
{%\color{blue}
8A'%
}
\sqrt{d}r_1)$ intersects 
either
$\Omega^c$, or 
%$B(x_1,\sqrt{d} r_1)$ intersects 
some obstacle $R$ in $\mathcal{R}_{\nseq,l}$ for $l \leq k$  with $\diam(R) \geq \beta r_1$. There can be at most one such obstacle, by condition (4) in Definition \ref{def:sparsecoll} and 
by the above 
choice of $k$. %we have $d(R,R') \geq \delta s_{k-1} \geq 4\sqrt{d} r_1$.

If there is no such obstacle, then define $\Omega'=\Omega$, and if there is one, then define $\Omega'=\Omega \setminus R$; %. This domain 
in either case, $\Omega'$
is $A'$-uniform %for some $A'=A'(A,D,\delta,L)$ by Theorem \ref{thm:cutout}.
with $A'$ fixed as above. As noted at the end of Step II, $\Omega'$ is also $C_{AR}$-Ahlfors regular and $D$-doubling.  Denote by $\mu'$ the restricted Lebesgue measure on $\Omega'$ for which we have $$
\Theta_{\mu'}(E,B(x_1,8A'\sqrt{d}r_1)) \geq (8A'\sqrt{d})^{-d} C_{AR}^{-2} \eta_1.
$$
Here, we used estimate \eqref{eq:r_1lowerbound} and the Ahlfors regularity of $\Omega'$.
For this domain,
{%\color{blue}
applying
Lemma \ref{lem:containedballs} 
with $ (8A'\sqrt{d})^{-d} C_{AR}^{-2}\eta_1$ for $\eta$,%
}
there %is 
exist
$\sigma = \sigma(D,A',\eta_1)$ and 
$x_2 = y \in B(x_1,4A' r_1)\cap \Omega'$ so that
%ball 
$B(x_2,\sigma r_1) \subset \Omega'\subset \Omega$ 
%with $x_2 \in B_\nseq(x_1,4A' r_1)$, so that 
as well as
$$
\Theta_{\lambda}(E,B(x_2,\frac{\sigma}{4\sqrt{d}} r_1))=\Theta_{\mu'}(E,B(x_2,\frac{\sigma}{4\sqrt{d}} r_1)) \geq 
\frac{1}{2 (8A'\sqrt{d})^{d} C_{AR}^{2}D^2}\eta_1.
$$
Let $r_2 = 
%\sigma r_1/(4\sqrt{d})
\frac{\sigma}{4\sqrt{d}} r_1
$ and put $\eta_2 = \frac{1}{2 (8A'\sqrt{d})^{d} C_{AR}^{2}D^2}\eta_1$,
{%\color{blue}
from which Equation \eqref{eq:goodballdensity} also follows;%
}
it now suffices to take $S_2 := \max\{4A',\frac{4\sqrt{d}}{\sigma}\}$.
%[{\color{red}TODO: We need $B_\nseq(y,{\color{red}\sigma}r_1) \subset \Omega'$, and the density bound, above. So we need to adjust Lemma \ref{lem:containedballs} a bit. ALSO, CHANGE ABOVE $x_1$ BALLS TO RADIUS WITH $A'$, %AS WELL AS $x_2$ BALL BELOW. It is because the $x_2$ ball could intersect a new obstacle, since $x_2 \neq x_1$ (and $\Omega' \neq \Omega$).  May need to move $A'$ earlier. ALSO, MENTION THE $S$ PARAMETER AS IN THEOREM STATEMENT.}]
With this choice of ball, we will have
$$
\frac{\sigma \delta}{%16 
{%\color{blue}
64A'
}
d} s_{k} \leq 
r_2 \leq 
\frac{\delta}{
{%\color{blue} 
16A'%
}
\sqrt{d}} s_{k-1}.
$$

\vspace{.2cm}
\noindent \textbf{Step IV:\ An isoperimetric estimate for %$B_\nseq(x_2,r_2)$
a good ball:} 
By construction, %we have that
$B(x_2, 2
{%\color{blue}
A'%
}
\sqrt{d} r_2)$ 
does not intersect $\Omega^c$. %, and for
{%\color{blue}
From our choice of $r_2$ above,%
}
any obstacle $R$ that intersects $B(x_2, 2
{%\color{blue}
A'%
}
\sqrt{d} r_2)$ 
{%\color{blue}
will also intersect $B(x_1,8A'\sqrt{d}r_1)$ and therefore satisfies}
either $R \in \mathcal{R}_{\nseq,l}$ for some $l>k$ or 
$R \in \mathcal{R}_{\nseq,k}$ with
\begin{equation} \label{eq:obstaclediameter}
\diam(R) \leq 
\beta r_1 \leq 
\frac{4\sqrt{d}}{\sigma}
\beta 
r_2%
.
\end{equation}
If there is such an obstacle, %let $R_0$ be it, 
then call it $R_0$;
otherwise let $R_0 = \emptyset$.

Now, let $Q$ be the cube centered at $x_2$ of side length $2r_2$ 
{%\color{blue}
and with faces parallel to the coordinate planes, so $Q$ contains%
}
the ball $B(x_2,r_2)$ %Since $S_{\nseq}$ is $d$-Ahlfors regular with constant $C_{AR}$, 
and
%Then, 
$\lambda(Q)\leq 2^d \sqrt{d}^d 
\omega_d^{-1}
\lambda(B(x_2,r_2)).$%e have 
%$$
%\lambda(Q) \leq 
%\lambda(B(x_2,\sqrt{d}r_2)) \leq 
%\omega_d \sqrt{d}^d r_2^d  \leq  
%{%\color{blue}
%\omega_d \sqrt{d}^d C_{AR}
%}\mu(B_\nseq(x_2,r_2)).
%$$

Then for $\eta_3=  (2^d \omega_d \sqrt{d}^d )^{-1} \eta_2$, 
the above estimate with Equation \eqref{eq:goodballdensity} implies
$$\Theta_{\lambda}(E,Q) \geq %\frac{1}{
%\color{red}
%C_1
%} 
(\omega_d \sqrt{d}^d 2^d)^{-1} \Theta_{\lambda}(E,B_\nseq(x_2,r_2)) \geq \eta_3.$$

By Lemma \ref{lem:projection}, there is an $i$ so that
\begin{equation}\label{eq:proji}
\mathcal{H}_{d-1}(\pi_i(Q \cap \partial_{+,i} E)) \geq \frac{(2r_2)^{d-1} \eta_3}{d}.
\end{equation}
Fix a choice of such an index $i$.

Let $\mathcal{S}=\pi_i( Q \cap \bigcup_{l} \bigcup_{R \in \mathcal{R}_{\nseq,l}} R)$, %. This is 
i.e.\
the ``shadow'' of all of the obstacles intersecting $Q$ and removed from $\Omega$. Consider the portion of the boundary not shadowed by obstacles:
$$
\partial^iE := Q \cap \partial_{+,i} E \setminus \pi_i^{-1}(\mathcal{S}).
$$

If $z \in \partial^iE$, then since $z \in Q \cap \partial_{+,i} E$, there is a sequence of points $z^E_{j} \in E \cap l_{i,z}$ ($j\in \N$) and and a sequence in its complement $z^{E^c}_j \in E^c \cap I_{i,z}$ converging to $z$. These sequences lie in $S_{\nseq},$ since they are not shadowed by obstacles in $Q$. Indeed, $z \in \partial_\nseq E$ and we have shown
\begin{equation}\label{eq:ithcoordbounday}
\partial^iE \subset \partial_\nseq E \cap Q.
\end{equation}

It thus suffices to prove that $\mathcal{H}_{d-1}(\partial^iE)$ is greater than a multiple of $r_2^{d-1}.$ We will do this by estimating $\mathcal{H}_{d-1}(\pi_i(\partial^iE))$ from below. To do this we note
\begin{equation}\label{eq:inclusionith}
\pi_i(\partial^iE) = \pi_i(Q \cap \partial_{+,i} E) \setminus \mathcal{S}.
\end{equation}

Now, $\mathcal{S}$ %is the union 
consists
of two parts: $\pi_i(\bigcup_{l > k}\bigcup_{R \in \mathcal{R}_{\nseq, 
l
}} \pi_i(R \cap 
Q))$ and  $\pi_i(R_0
\cap Q
)$. Since $R_0$ is either empty 
%has $\diam(R_0) \leq \beta r_2$, then 
or satisfies Equation \eqref{eq:obstaclediameter}, it follows that
\begin{equation}\label{eq:R0proj}
\mathcal{H}_{d-1}(\pi_i(R_0)) \leq 
\beta^{d-1} \Big( \frac{4\sqrt{d}}{\sigma} \Big)^{d-1} r_2^{d-1}.
\end{equation}
Let $\rho = \frac{
%16
{%\color{blue}
64A'
}
d}{\sigma \delta}r_2 \geq s_k$, so $Q \subset B(x_2,\rho)$. For the other part, we apply condition (6) from Definition \ref{def:sparsecoll}:
\begin{eqnarray}\label{eq:shadowobs} 
\hskip1cm \mathcal{H}_{d-1} \Big(\pi_i\left(\bigcup_{l > k}\bigcup_{R \in \mathcal{R}_{\nseq, l}} \pi_i(R \cap B(x_2,\rho))\right)\Big) \leq %\nonumber \leq 
\sum_{l=k+1}^\infty \frac{L \rho^{d-1}}{n_l^{d-1}} \leq
\frac{\epsilon_1 L(%16
{%\color{blue}
64A'
}
d)^{d-1}}{\sigma^{d-1} \delta^{d-1}} r_2^{d-1}.
\end{eqnarray}

Then the sum of estimates \eqref{eq:R0proj} and \eqref{eq:shadowobs} gives
\begin{eqnarray}\label{eq:summary}
\Ha_{d-1}(\mathcal{S}) &\leq& \mathcal{H}_{d-1}(\pi_i(R_0)) + \mathcal{H}_{d-1} \left(\pi_i\left(\bigcup_{l > k}\bigcup_{R \in \mathcal{R}_{\nseq, l}} \pi_i(R \cap B(x_2,\rho))\right)\right) \nonumber \\
&\leq& \left( \frac{\epsilon_1 L(%16
{%\color{blue}
64A'
}
d)^{d-1}}{\sigma^{d-1} \delta^{d-1}} + \beta^{d-1}
{%\color{blue}
\Big(
\frac{4\sqrt{d}}{\sigma}
\Big)^{d-1}
}
\right) r_2^{d-1} .
\end{eqnarray}

Now, we choose $\epsilon_1 = \frac{\sigma^{d-1} \delta^{d-1} \eta_3 }{ 4 d  L(%16
{%\color{blue}
64A'
}
d)^{d-1} }$ and 
$\beta = 
{%\color{blue}
\frac{\sigma}{4\sqrt{d}}
}
(\frac{\eta_3}{4 d})^{\frac{1}{d-1}}$. These choices, together with estimates \eqref{eq:summary} and \eqref{eq:proji} together with the equality \eqref{eq:inclusionith} and inclusion \eqref{eq:ithcoordbounday} give

\begin{eqnarray}
\Ha_{d-1}(\partial_\nseq E \cap B(x_2,\sqrt{d} r_2)) ~\geq~
\Ha_{d-1}(\partial_\nseq E \cap Q)  &\geq& 
\Ha_{d-1}(\partial^i E) \nonumber \\
&\geq& 
\Ha_{d-1}(\pi_i(\partial^i E)) \nonumber \\
&\stackrel{\eqref{eq:inclusionith}}{\geq}& 
\Ha_{d-1}(\pi_i(Q \cap \partial_{+,i} E)) - \Ha_{d-1}(\mathcal{S}) \nonumber \\
&\stackrel{\tiny \begin{array}{c}
\eqref{eq:summary} \\
\eqref{eq:proji}
\end{array}}{\geq}&
\frac{(2r_2)^{d-1}\eta_3}{d} - \frac{r_2^{d-1}\eta_3}{2d} \geq \frac{\eta_3}{2d} r_2^{d-1}.
\end{eqnarray}

This is estimate \eqref{eq:reduction}, from which \eqref{eq:goal} follows after applying doubling, Ahlfors regularity as well as the estimates for $d(x_2,x) \leq 
Sr$ and $r_2 \geq \frac{1}{S} r$ obtained by following the previous steps. This concludes the proof of the isoperimetric inequality.
\end{proof}

\appendix
\section{Explicit examples}\label{app:explicit}

Here, we give an explicit application of Theorem \ref{thm:mainthm}. This is a generalization of a construction of Mackay, Tyson, and Wildrick \cite{mackaytysonwildrick} to higher dimensions. 

Let ${\bf n} = (n_i)_{i=1}^\infty$ be a %fixed 
sequence of odd positive integers with $n_i \geq 3$. 
Fix a dimension $d \geq 2$ and consider the following iterative construction:

\begin{enumerate}
\item
At the first stage,
put
$S_{0,\nseq}=[0,1]^d$ and 
$T_{0,\nseq}^1 = [0,1]^d$ and 
$\mT_{0,\nseq}=\{T_{0,\nseq}^1\}$.

\item
Assuming that we have defined
$S_{k,\nseq}, T_{k,\nseq}^j$,
$\mT_{k,\nseq}$ at the $k$th stage, 
for $k \in \N$,

\begin{itemize}
\item
Subdivide each $T \in \mT_{k,\nseq}$
into $(n_{k+1})^d$ equal subcubes and exclude the central one. 
\item
Index the remaining subcubes in any fashion as $T_{k+1, \nseq}^j$,
and let $\mT_{k+1, \nseq} = \{T_{k+1,\nseq}^j\}$ be the %$k+1$'th stage 
collection of all the 
remaining
cubes.
\end{itemize}
%\item 
The {\sc $k+1$'th order pre-sponge} 
is defined, consistent with \eqref{eq:presponge}, as the set 
$$
S_{k+1, \nseq} = \bigcup_{T \in \mT_{k+1, \nseq}} T.
$$

\item
Define $\mR_{\nseq, k}$ to be the set of central $1/n_{k+1}$ subcubes that were removed from each $T \in \mT_{k,\nseq}$ which are removed at the $k$'th stage %, for $k=1, \dots, \infty$. Also, define 
and put
$\nmR_{\nseq, k}=\bigcup_{l=1}^k \mR_{\nseq,l}$.
Further we note that for $k \in \N$,
%\sout{define}
$$
s_k = \prod_{i=1}^k \frac{1}{n_i}
$$
%\sout{to be} {\edit is} 
is the side length of each cube $T \in \mathcal{T}_{k,\nseq}$. %$k$'th level cube for $k \geq 1$. %\sout{and also,} 
(For consistency, let $s_0=1$.)
\end{enumerate}

The {\sc Sierpi\'nski-sponge associated to the sequence $\nseq$} is then defined as
\begin{equation}
S_{\nseq} = \bigcap_{k \geq 0} S_{k,\nseq}. \label{eq:sierdef}
\end{equation}
When $d=2$, we often call these Sierpi\'nski-carpets due to the fact when
%and for $d=2$ %When $d=2$ we also refer to these sets as Sierpi\'nski carpets,  and 
 $\nseq=(3,3,3\ldots)$ the construction yields the usual ``middle-thirds'' Sierpi\'nski carpet. Indeed, in the plane, all $S_{\nseq}$ are homeomorphic to this space.

The main results by Mackay, Tyson, and Wildrick \cite[Theorem 1.5--1.6]{mackaytysonwildrick}, for dimension $d=2$, characterizes when a Sierpi\'nski carpets with the restricted measure satisfies a Poincar\'e inequality. Specifically, the subset satisfies a $1$-Poincar\'e inequality if and only if  $\mathbf{n}^{-1}=(\frac{1}{n_i}) \in \ell^1$. Their result also states that the carpet satisfies a $p$-Poincar\'e inequality for some (or any) $p>1$ if and only if $\mathbf{n}^{-2}=(\frac{1}{n_i^2}) \in \ell^2$. %These different inequalities are obtained from Definition \ref{def:PI} by replacing the right hand side by an $L^p$-norm. 
The $p>1$ regime behaves quite differently, and the authors investigated this in more detail in a separate paper \cite{bigpaper}. In that paper, there also appears a version of the following theorem for $p>1$; see \cite[Theorem 1.6]{bigpaper}. These together fully extend the results \cite{mackaytysonwildrick} to all dimension, as well as to obstacles with different geometries.

The following 
result is a
higher dimensional analogue of 
the Mackay-Tyson-Wildrick Theorem, for the $p=1$ case.

\begin{theorem}\label{thm:ndimSiercarpet} Let ${\nseq} = (n_i)$ be a sequence of odd integers with $n_i \geq 3$, and let $d \geq 2$.  The space $(S_{\nseq}, |\cdot|, \lambda)$ satisfies a $1$-Poincar\'e inequality if and only if %${\bf n}^{-1} \in \ell^{d-1}(\N)$.
\begin{align} \label{eq:d-1summable}
\sum_{i=1}^\infty \frac{1}{n_i^{d-1}} < \infty.
\end{align}
\end{theorem}

\begin{proof} We check the various conditions of having a sparse collection of co-uniform domains in $[0,1]^d$ with small projections and then the claim follows from Theorem \ref{thm:mainthm}. These can be directly verified with the choicses $L=4(\sqrt{d}r)^{d-1},\delta=\frac{1}{3}$, $\Omega=T_0$ and $A=\frac{1}{6d}.$

\vspace{.2cm}

\noindent \textbf{Conditions (1), (2):} The uniformity and co-uniformity of squares is easy. For example, for the unit cube $T_0$, if we take $x,y \in T_0=\Omega$ with $d(x,y)=s$, then we can form $\gamma$ by first choosing  radial paths towards the center of the square $c$ from $x$ and $y$ of length $\min\{s/2,d(x,c)\}$ and $ \min\{s/2,d(y,c)\}$, respectively, and then concatenating by a straight line segment. This gives a path of length at most $6s$, which is $\frac{1}{6d}$-uniform.

\vspace{.2cm}

\noindent \textbf{Condition (3):} Each $R \in \mathcal{R}_{\nseq,k}$ has side length $s_k$, and so diameter at most $\sqrt{d} s_k$.

\vspace{.2cm}

\noindent \textbf{Conditions (4), (5):} If $R \in \mathcal{R}_{\nseq,k}$, then $R$ is a central $1/n_k$ cube of some $T \in \mathcal{T}_{\nseq,k-1}$. Thus the distance to the boundary, or any other higher level obstacle, is at least the distance to the boundary of $T$, that is at least $\frac{1}{3}s_{k-1}$, since $n_{k} \geq 3.$ 

\vspace{.2cm}

\noindent \textbf{Condition (6):} Set $K=1$. Include $\pi_i(B(x,r) \cap S_{0,\nseq})$ into a $d-1$ dimensional cube $Q_0$ of side length at most $4 \sqrt{d} r$, which is a union of $d-1$ cubes $\mathcal{Q}=\{Q_i : i=1,\dots, N\}$, for some $N$, in the grid of side length $s_{k-1}$. Each cube $Q_i$ of side length $s_{k-1}$ will include at most one projected cube $\pi_i(R)$ for some $R \in \mathcal{R}_{\nseq,k}$. Such a cube is centered and accounts for at most $\frac{1}{n_k^{d-1}}$ of the volume $\lambda(Q_i)$. Thus, $\lambda\left(\pi_i\left(\bigcup_{R\in \mathcal{R}_{\nseq,k}} \cap Q_i\right)\right) \leq \frac{1}{n_k^{d-1}} \lambda(Q_i)$. The entire volume of the cube $Q_0$ is at most $(4\sqrt{d}r)^{d-1}$. Summing over all $i=1,\dots, N$ gives the claim with $L= (4\sqrt{d}r)^{d-1}$.

\end{proof}

\begin{remark}\label{rmk:triangles}One can replace the square lattice with a triangular lattice, and perform the removal procedure on a central triangle instead of a central square. The only crucial property one must ensure is that the central triangle does not intersect the boundary of its parent triangle. The triangle lattice also comes with two natural (non-orthogonal) projections and one may verify the conditions of Theorem \ref{thm:mainthm} in a similar fashion. 
This would give a triangular version, depicted in Figure \ref{fig:sierpinskitriangle}, of Theorem \ref{thm:ndimSiercarpet}. The details are left to the reader.
\end{remark}

\bibliographystyle{amsplain}

\vfill
\end{document}